\documentclass[11pt]{amsart}

\usepackage{fullpage}
\usepackage{amsmath}
\usepackage{amsthm}
\usepackage{amsfonts}
\usepackage{amssymb}
\usepackage{graphicx}

\newcommand{\mb}{\mathbf}
\newcommand{\mc}{\mathcal}
\renewcommand{\Re}{\mathrm{Re}}

\newtheorem{lemma}{Lemma}[section]
\newtheorem{proposition}[lemma]{Proposition}
\newtheorem{theorem}[lemma]{Theorem}
\newtheorem{corollary}[lemma]{Corollary}
\theoremstyle{remark}
\newtheorem{remark}[lemma]{Remark}

\theoremstyle{definition}
\newtheorem{definition}[lemma]{Definition}

\title{On stable self--similar blow up for equivariant wave maps: The linearized problem}

\author{Roland Donninger}
\address{University of Chicago, Department of Mathematics,
5734 South University Avenue, Chicago, IL 60637, U.S.A.}
\email{donninger@uchicago.edu}
\thanks{The first author 
is an Erwin Schr\"odinger Fellow of the 
FWF (Austrian Science Fund) Project No. J2843. He wants to thank Wilhelm Schlag for many interesting discussions and the ESI Vienna for hospitality during the workshop ``Quantitative studies of nonlinear wave phenomena''.}

\author{Birgit Sch\"orkhuber}
\address{Vienna University of Technology, Faculty of Mathematics and Geoinformation, Institute for Analysis and Scientific Computing,  Wiedner Hauptstr.~8, 1040 Wien, Austria}             
\email{schoerkhuber@asc.tuwien.ac.at}
\thanks{The second author acknowledges partial support from
the program ``Innovative Ideas'' of Vienna University of Technology;
the Austria-Croatia Project HR 10/2010;
the Austria-France Project FR 07/2010;
and the Austria-Spain Project ES 08/2010 of the Austrian Exchange Service
(\"OAD)}

\author{Peter C. Aichelburg}
\address{Universit\"at Wien, Fakult\"at f\"ur Physik, Gravitational Physics,
Boltzmanngasse 5, A-1090 Wien, Austria}
\email{aichelp8@univie.ac.at}
\thanks{The third author acknowledges partial support from the FUNDACION FEDERICO 
and the Erwin Schr\"odinger Institute for Mathematical Physics, Vienna, for
hosting a program on "Quantitative studies of nonlinear wave phenomena".}

\begin{document}

\begin{abstract}
We consider co--rotational wave maps from ($3+1$) Minkowski space into the three--sphere. This is an energy supercritical model which is known to exhibit finite time blow up via self--similar solutions. The ground state self--similar solution $f_0$ is known in closed form and based on numerics, it is supposed to describe the generic blow up behavior of the system. 
In this paper we develop a rigorous linear perturbation theory around $f_0$.
This is an indispensable prerequisite for the study of nonlinear stability of the self--similar blow up which is conducted in the companion paper \cite{wmnlin}.
In particular, we prove that $f_0$ is linearly stable if it is mode stable.
Furthermore, concerning the mode stability problem, we prove new results that exclude the existence of unstable eigenvalues with large imaginary parts and also, with real parts larger than $\frac{1}{2}$.
The remaining compact region is well--studied numerically and all available results strongly suggest the nonexistence of unstable modes.
\end{abstract}

\maketitle

\section{Introduction}
Let $(M,g)$ be a Lorentzian manifold and $(N,h)$ a Riemannian manifold with metrics $g$ and $h$, respectively.
Consider the geometric action functional
$$ S(u)=\int_M \mathrm{tr}_g (u^*h), $$
for a function $u: M \to N$, i.e., $S(u)$ is the integral over the trace (with respect to $g$) of the pullback $u^* h$.
In abstract index notation, this expression reads
$$ S(u)=\int_M g^{\mu \nu} (\partial_\mu u^a) (\partial_\nu u^b) h_{ab}\circ u $$
and \emph{formally}, in local coordinates $(x^\mu)$, the Euler--Lagrange equations associated to this action are given by
\begin{equation}
\label{eq:wm}
\Box_g u^a(x)+g^{\mu \nu}(x)\Gamma^a_{bc}(u(x))\partial_\mu u^b(x) \partial_\nu u^c(x)=0
\end{equation}
where Einstein's summation convention is in force. The indices $\mu,\nu$ and $a,b,c$ take the values $0,1,\dots,\dim M-1$ and $1,2,\dots,\dim N$, respectively, and $\Box_g$ denotes the Laplace--Beltrami operator on $M$.
Furthermore, $\Gamma^a_{bc}$ are the Christoffel symbols associated to the metric $h$ on the target manifold $N$.
Eq.~\eqref{eq:wm} is known as the \emph{wave maps equation} in the intrinsic formulation.
In order to develop some intuition, let us consider simple examples. 
If $N=\mathbb{R}^n$ with the standard Euclidean metric and $M$ is $(m+1)$--dimensional Minkowski space with coordinates $(t,x^1,x^2,\dots,x^m)$, then Eq.~\eqref{eq:wm} reduces to
$$ u^a_{tt}-\Delta_{\mathbb{R}^m}u^a=0, $$
the ordinary linear wave equation for a vector--valued function.
As a consequence, if the target is curved, the wave maps equation constitutes a nonlinear generalization of the wave equation and its geometric origin makes this generalization very natural.
Indeed, consider for instance $M=(0,1)$, an open interval, and $N$ any Riemannian manifold.
Then the corresponding wave maps equation reads
$$ \frac{d^2 u^a}{dt^2}(t)+\Gamma^a_{bc}(u(t))\frac{du^b}{dt}(t)\frac{du^c}{dt}(t)=0 $$
which is nothing but the geodesic equation on $N$. 
These simple examples already indicate that the wave maps functional is a rich source for interesting nonlinear equations.

In physics, wave maps first appeared in the 1960s as the SU(2) sigma model \cite{Gell-Mann} which corresponds to mappings from physical (that is, $(3+1)$--dimensional) Minkowski space to the three--sphere.
In fact, the present work deals exactly with this model subject to a certain symmetry reduction, see below.
Let us mention at least two more applications of wave maps in physics.
Due to its geometric nature, the wave maps equation is frequently used as a toy model for the much more involved Einstein equations of general relativity, in particular in the context of critical gravitational collapse. 
Furthermore, self--gravitating wave maps are used in general relativity to study the behavior of nonlinear matter fields coupled to Einstein's equations.
These works include the study of singularity formation (``hairy black holes'', ``naked singularities''), solitons with nonzero cosmological constant and, more recently, long--time asymptotics (``tails'') of global solutions (see e.g., \cite{clement-fabbri}, \cite{lechner-husa-pca}, \cite{bizonwm} and references therein). 
We also refer the reader to \cite{misner} for more (possible) applications of wave maps in physics.

From now on we restrict ourselves to wave maps on Minkowski space, i.e., we assume that $M=\mathbb{R}^{m+1}$ with the standard Minkowski metric $(g_{\mu \nu})=\mathrm{diag}(-1,1,\dots,1)$ and coordinates $(t,x^1,x^2,\dots,x^m)$.
In this case, Eq.~\eqref{eq:wm} reads
\begin{equation}
\label{eq:wm2} u^a_{tt}(t,x)-\Delta_{\mathbb{R}^m}u^a(t,x)=g^{\mu \nu}\Gamma^a_{bc}(u(t,x))\partial_\mu u^b(t,x)\partial_\nu u^c(t,x)
\end{equation}
and it is natural to consider the Cauchy problem, that is, one prescribes initial data $(u^a(0,\cdot),u_t^a(0,\cdot))$ and studies the future development.
The Cauchy problem for the wave maps equation has attracted a lot of interest in the past 20 years.
Eq.~\eqref{eq:wm2} is a system of semilinear wave equations and the standard theory immediately yields local well--posedness in the Sobolev space $H^s(\mathbb{R}^m) \times H^{s-1}(\mathbb{R}^m)$ provided that $s>\frac{m}{2}+1$, see e.g., \cite{klainerman-selberg}.
However, by exploiting the null structure of the wave maps nonlinearity, it is possible to improve this result to $s>\frac{m}{2}$ (``Klainerman--Machedon theory''), see the survey \cite{klainerman-selberg} and references therein. 
The wave maps equation \eqref{eq:wm2} is invariant under the scaling $u(t,x) \mapsto u_\lambda(t,x):=u(\lambda t,\lambda x)$ for a $\lambda>0$ and there exists 
a conserved energy given by
$$ E(u)=\frac{1}{2}\int_{\mathbb{R}^m} \left [h_{ab}(u(t,x))\partial_t u^a(t,x)\partial_t u^b(t,x)+h_{ab}(u(t,x))\sum_{i=1}^m \partial_i u^a(t,x)\partial_i u^b(t,x) \right ]dx $$
which scales like $E(u_\lambda)=\lambda^{2-m}E(u)$.
This shows that the wave maps equation is energy subcritical if $m=1$, critical if $m=2$ and supercritical if $m\geq 3$.
Based on a well--known heuristic principle one expects large data global well--posedness in the energy subcritical case and finite time blow up for energy supercritical equations.
Indeed, large data global well--posedness in the subcritical case has been proved in \cite{keel-tao}.
Also, small data global well--posedness in dimensions $m \geq 2$ for a large class of targets has been established in the last ten years, see e.g., \cite{tataru99}, \cite{tataru01}, \cite{tao01a}, \cite{tao01b}, \cite{klainerman-rodnianski02}, \cite{shatah-struwe02}, \cite{tataru05}, \cite{krieger03}, \cite{krieger04}, \cite{nahmod03}, \cite{nahmod02}.
We refer the reader to \cite{kriegersurv} for a survey on these results and a detailed list of references.
In the energy critical dimension $m=2$, the questions of large data global existence and blow up are related to the geometry of the target.
For spherical targets (and also more general surfaces of revolution), blow up solutions have been constructed in \cite{KST08}, \cite{rodnianski-sterbenz06} and, very recently, \cite{carstea}, \cite{rodnianski-raphael09}.
Again, we refer the reader to \cite{kriegersurv} for more references and earlier results related to this.
On the other hand, large data global well--posedness results for the energy critical case have been obtained in \cite{krieger-schlag09} as well as \cite{tataru-sterbenz09a}, \cite{tataru-sterbenz09b} and in the series \cite{tao09}.

In the present paper we study the energy supercritical case. More precisely, we consider wave maps from physical Minkowski space to the three--sphere, the original SU(2) sigma model of particle physics.
Furthermore, we impose a restrictive symmetry assumption: we require the wave map to be \emph{co--rotational}.
In order to explain what this means, choose standard spherical coordinates $(t,r,\theta,\varphi)$ on Minkowski space and hyperspherical coordinates $(\psi,\Theta,\Phi)$ on the three--sphere, i.e., the respective metrics are given by
$$ g=-dt^2+dr^2+r^2 (d\theta^2+\sin^2 \theta d\varphi^2) $$
and
$$ h=d\psi^2+\sin^2 \psi(d\Theta^2+\sin^2 \Theta d\Phi^2). $$
In these coordinates, a mapping $u: M \to N$ is described by the three functions $\psi(t,r,\theta,\varphi)$, $\Theta(t,r,\theta,\varphi)$ and $\Phi(t,r,\theta,\varphi)$.
The mapping $u$ is said to be co--rotational if $\Theta$ and $\Phi$ are trivial in the sense that $\Theta(t,r,\theta,\varphi)=\theta$, $\Phi(t,r,\theta,\varphi)=\varphi$ and $\psi$ is independent of $\theta$, $\varphi$.
The wave maps equation for co--rotational maps reduces to the single semilinear wave equation
\begin{equation}
\label{eq:main}
\psi_{tt}-\psi_{rr}-\frac{2}{r}\psi_r+\frac{\sin(2\psi)}{r^2}=0
\end{equation}
and requiring regularity (e.g., $C^2$) of $\psi$ at the origin yields the asymptotic behavior $\psi(t,r)=O(r)$ for any $t$ as $r \to 0+$ and in particular we have the boundary condition $\psi(t,0)=0$ for all $t$.
However, note carefully that smoothness of $\psi$ at $r=0$ does not imply that $\psi_r(t,0)$ has to vanish (observe the special cancellation by Taylor expansion).

A considerable amount of the wave maps literature is devoted to the study of the Cauchy problem in the presence of symmetries such as equivariance or spherical symmetry, see e.g., \cite{Chr1}, \cite{Chr2}, \cite{struwe03}, \cite{struwe04a}, \cite{struwe04b}, \cite{struwe99}, \cite{struwe01}, \cite{struwe98}. For the model 
under investigation, the papers \cite{sideris} and \cite{STZ94} are of particular interest.
In \cite{sideris}, small data global well--posedness in a sufficiently high Sobolev space is established whereas in \cite{STZ94}, many different aspects of the Cauchy problem are considered, in particular the minimal regularity requirements for local well--posedness.
It is well--known that Eq.~\eqref{eq:main} exhibits self--similar finite time blow up.
A solution $\psi^T$ of \eqref{eq:main} is said to be self--similar if it is of the form $\psi^T(t,r)=f(\frac{r}{T-t})$ for a constant $T>0$ and a smooth function $f$ \footnote{A priori, a self--similar solution is of the form $\psi(t,r)=f(\frac{r}{t})$ but one may immediately apply the time translation and reflection symmetries of the equation to obtain the one--parameter family $\psi^T(t,r)=f(\frac{r}{T-t})$.}.
The existence of a smooth self--similar solution $f_0$ for Eq.~\eqref{eq:main} has been proved in \cite{shatah88} by variational techniques and independently, it has been found in closed form \cite{TS90}, see also \cite{cazenave}.
We refer to $f_0$ as the \emph{fundamental} or \emph{ground state self--similar solution} and reserve the symbol $\psi^T$ for this solution, i.e., 
$$ \psi^T(t,r):=f_0\left (\tfrac{r}{T-t}\right )=2\arctan \left (\tfrac{r}{T-t} \right ). $$
By exploiting finite speed of propagation, $f_0$ can be used to construct a perfectly smooth solution of Eq.~\eqref{eq:main} with compactly supported initial data that breaks down at $t=T$.
However, a natural question to ask is how generic this break down is.
Does it only happen for very special initial data or is there an ``open'' set of data that lead to this type of self--similar blow up?
Numerical experiments \cite{bizon99} indicate that the latter is true, i.e., the blow up described by $f_0$ is conjectured to be stable.
More precisely, one observes that the future development of sufficiently large generic initial data converges to $\psi^T$ in the backward lightcone of the spacetime point $(T,0)$ as $t \to T-$.

It is worth mentioning that the blow up behavior of the corresponding energy critical model of equivariant wave maps on $(2+1)$ Minkowski space is fundamentally different.
Due to a result of Struwe \cite{struwe03}, it is known that the blow up in the energy critical case
cannot be self--similar.
In this respect it is interesting to note that
Struwe's result does not imply that blow up actually happens.
However, as already mentioned, blow up solutions for equivariant wave maps from (2+1) Minkowski space to the two--sphere have been constructed in \cite{KST08}, \cite{rodnianski-sterbenz06}, \cite{rodnianski-raphael09}.

\subsection{Outline of the paper}
In the present paper we develop a functional analytic linear perturbation theory around the self--similar solution $\psi^T$.
Since $\psi^T$ is time--dependent, we introduce adapted coordinates $(\tau,\rho)$ (``similarity variables'') such that $\psi^T$ becomes independent of the new time coordinate $\tau$.
The coordinates $(\tau,\rho)$ cover the backward lightcone of the blow up point $(t,r)=(T,0)$ which is mapped to $\tau=\infty$.
The linearization of Eq.~\eqref{eq:main} around the solution $\psi^T$ in the adapted coordinates is written as a system of the form
\begin{equation} 
\label{eq:intro}
\frac{d}{d\tau}\Phi(\tau)=L\Phi(\tau)  
\end{equation}
where $L$ is a linear spatial differential operator which is realized as an unbounded operator on a suitable Hilbert space.
However, the linearized operator $L$ is not self--adjoint (in fact, not even normal) and thus, the analysis of this equation is highly nontrivial since one cannot resort to standard methods from self--adjoint spectral theory.
Consequently, we apply semigroup theory to study the linearized time evolution of perturbations of $\psi^T$.
In Sec.~\ref{sec:ss}, we present a heuristic discussion that motivates the choice of the underlying Hilbert space.
Then, in Sec.~\ref{sec:mode}, we review known theoretical and numerical results on the point spectrum of $L$, that is, we consider mode solutions $\Phi(\tau)=e^{\lambda \tau}\mb{u}$ of Eq.~\eqref{eq:intro} and prove a new result that excludes unstable eigenvalues with real parts larger than $\frac{1}{2}$.
In Sec.~\ref{sec:wp}, we show that $L$ generates a strongly continuous one--parameter semigroup $S(\tau)$.
This yields the well--posedness of the Cauchy problem for Eq.~\eqref{eq:intro} and provides the rigorous spectral theoretic basis for the mode analysis in Sec.~\ref{sec:mode}.
It turns out that the time translation symmetry of the wave maps equation induces a single unstable eigenvalue in the spectrum of $L$.
This introduces an additional difficulty since the instability is ``artificial'' and one is actually interested in the flow ``modulo this instability''.
In order to make this precise, we construct an appropriate spectral projection that decomposes the underlying Hilbert space in a stable and an unstable part.
We show that the unstable subspace is spanned by the single eigenvector associated to the aforementioned unstable eigenvalue of the linearized operator $L$.
Finally, we prove an appropriate growth bound for the linear time evolution $S(\tau)$ restricted to the stable subspace which shows that $\psi^T$ is indeed linearly stable, provided that there are no unstable modes. As a consequence, we obtain a thorough understanding of the linear stability problem for $\psi^T$. 
Remarkably, we identify many features that are known from analogous self--adjoint problems.
As a by--product, we also show that unstable eigenvalues with large imaginary parts do not occur. This result justifies some of the numerical methods that have been employed to study mode stability of $\psi^T$, in particular the shooting method.
The results of this paper are of fundamental importance for the study of the full nonlinear stability of $\psi^T$ which is conducted in \cite{wmnlin}.

\subsection{Notations and conventions}
For Banach spaces $X,Y$ we denote by $\mc{B}(X,Y)$ the Banach space of bounded linear operators from $X$ to $Y$. As usual, we write $\mc{B}(X)$ if $X=Y$.
Vectors are denoted by bold letters and the individual components are numbered by lower indices, e.g., $\mb{u}=(u_1,u_2)$. We do not distinguish between row and column vectors.
Furthermore, we denote the spectrum and resolvent set of a linear operator $A$ by $\sigma(A)$ and $\rho(A)$, respectively.
In particular, we write $\sigma_p(A)$ for the point spectrum, i.e., the set of all eigenvalues.
For $\lambda \in \rho(A)$ 
we set $R_A(\lambda):=(\lambda-A)^{-1}$, i.e., $R_A$ is the resolvent of $A$.
Finally, for $a,b \in \mathbb{R}$ we use the notation $a \lesssim b$ if there exists a constant $c>0$ such that $a \leq cb$ and we write $a \simeq b$ if $a \lesssim b$ and $b \lesssim a$. 
The symbol $\sim$ is reserved for asymptotic equality.
Also, the letter $C$ (possibly with indices) denotes a generic nonnegative constant which is not supposed to have the same value at each occurrence.

\section{Self--similar blow up}
\label{sec:ss}
\label{sec:bunorm}
Our strategy is to treat the nonlinearity in Eq.~\eqref{eq:main} as a perturbation of a suitable ``free'' equation.
Note that the nonlinearity is singular at $r=0$ and thus, we have to split off the singular part first and assign it to the differential operator, i.e., the free equation.
In order to do so, we simply rewrite Eq.~\eqref{eq:main} as
\begin{equation} 
\label{eq:mainreg}
\psi_{tt}-\psi_{rr}-\frac{2}{r}\psi_r+\frac{2}{r^2}\psi+\frac{\sin(2\psi)-2\psi}{r^2}=0.  
\end{equation}
Formally, the nonlinear term $\frac{\sin(2\psi)-2\psi}{r^2}$ is now regular at $r=0$ since we assume $\psi(t,r)=O(r)$ as $r \to 0+$.
Our goal is to study Eq.~\eqref{eq:mainreg} in the backward lightcone 
$$\mc{C}_T:=\{(t,r): t\in (0,T), r \in [0,T-t]\}$$ 
of the blow up point $(t,r)=(T,0)$.
The conserved energy of the \emph{free equation}
\begin{equation}
 \label{eq:free}
\psi_{tt}-\psi_{rr}-\frac{2}{r}\psi_r+\frac{2}{r^2}\psi=0
\end{equation}
is given by
$$ \frac{1}{2}\int_0^\infty \left (r^2 \psi_t^2(t,r)+r^2 \psi_r^2(t,r)+2\psi^2(t,r) \right )dr $$
as follows immediately by multiplying the equation by $r^2 \psi_t$ and integrating by parts.
Naively one might expect this energy to yield a natural norm suitable for the study of the stability of $\psi^T$.
We give an argument why this is not the case.
Consider the local energy in the backward lightcone $\mc{C}_T$ given by
$$ E_\psi(t):=\frac{1}{2}\int_0^{T-t} \left (r^2 \psi_t^2(t,r)+r^2 \psi_r^2(t,r)+2\psi^2(t,r) \right )dr. $$
By energy conservation, one has the estimate $E_\psi(t) \lesssim 1$ and for a generic energy solution living on $\mc{C}_T$ this \emph{cannot} be improved to obtain polynomial decay of the form $E_\psi(t) \lesssim (T-t)^\gamma$ as $t \to T-$ for some constant $\gamma>0$. To see this, consider the function
$$\psi_\varepsilon(t,r)=(T-t)^{-\frac{1}{2}+\varepsilon}u_\varepsilon\left (\tfrac{r}{T-t} \right ) $$ where
$$ u_\varepsilon(\rho):=\frac{1-(\frac{1}{2}+\varepsilon)\rho}{\rho^2}(1+\rho)^{\frac{1}{2}+\varepsilon}
-\frac{1+(\frac{1}{2}+\varepsilon)\rho}{\rho^2}(1-\rho)^{\frac{1}{2}+\varepsilon} $$
for arbitrary $\varepsilon>0$
which
is a solution of the free equation inside the lightcone $\mc{C}_T$.
Its local energy satisfies $E_{\psi_\varepsilon}(t)\simeq (T-t)^{2\varepsilon}$ and, since $\varepsilon>0$ is arbitrary,
this shows that there cannot exist a $\gamma>0$ such that $E_\psi(t)\lesssim (T-t)^\gamma$ as $t \to T-$ for generic solutions $\psi$ of the free equation.
On the other hand, for the self--similar function $\psi^T$ we have
$E_{\psi^T}(t) \lesssim (T-t)$ and thus, \emph{the local energy of $\psi^T$ decays faster than the local energy of a generic free solution}.
This, of course, renders a perturbative approach hopeless from the very beginning.
As a consequence, the local energy does not yield a suitable norm for studying the blow up of $\psi^T$ because \emph{it does not blow up}.
Obviously, the same is true for the energy of the full nonlinear equation \eqref{eq:main}.
Thus, in order to make a perturbative approach feasible, we need to find a norm $\|\cdot\|$ which is derived from a conserved quantity of the free equation and
$\|(\psi^T(t,\cdot),\psi^T_t(t,\cdot))\| \to \infty$ as $t \to T-$.
The idea is obvious: we have to consider the gradient instead of the function itself since hitting the self--similar solution $\psi^T$ with $\partial_r$ brings out the singular factor $(T-t)^{-1}$.
Thus, we need a convenient equation for $\psi_r$.
However, differentiating Eq.~\eqref{eq:free} directly yields a true third--order equation for $\psi_r$ due to the term $\frac{2}{r^2}\psi$ which is not desirable.
Hence, we have to remove this term first by transforming to a new dependent variable $\tilde{\psi}$ of the form $\psi(t,r)=r^k \tilde{\psi}(t,r)$ for a suitable $k$.
We obtain
$$ \tilde{\psi}_{tt}-\tilde{\psi}_{rr}-\frac{2(k+1)}{r}\tilde{\psi}_r-\frac{(k-1)(k+2)}{r^2}\tilde{\psi}=0 $$
and this shows that there are two possibilities, either $k=1$ or $k=-2$. 
Let us investigate those two possible paths.
At first glance one might expect $k=1$ to be the more appealing choice since it leads to the radial $5$--dimensional wave equation
$$ \tilde{\psi}_{tt}-\tilde{\psi}_{rr}-\frac{4}{r}\tilde{\psi}_r=0 $$
and differentiating with respect to $r$ yields
$$ \tilde{\psi}_{rtt}-\tilde{\psi}_{rrr}-\frac{4}{r}\tilde{\psi}_{rr}+\frac{4}{r^2}\tilde{\psi}_r=0. $$ 
Multiplying by $r^4 \tilde{\psi}_{rt}$ and integrating by parts we obtain the conserved quantity
\begin{equation}
\label{eq:hienwrong} \frac{1}{2}\int_0^\infty \left (r^4 \tilde{\psi}_{rt}(t,r)^2+r^4 \tilde{\psi}_{rr}(t,r)^2+4 r^2 \tilde{\psi}_r(t,r)^2 \right )dr. 
\end{equation}
This suggests to use $\tilde{\psi}_t$ and $\tilde{\psi}_r$ as evolution variables.
However, there is a serious disadvantage of this choice: we do not have a boundary condition for $\tilde{\psi}(t,r)$ (recall that $\tilde{\psi}=\frac{\psi}{r}$ and $\psi(t,r)=O(r)$ as $r \to 0+$).
As a consequence, \emph{at a fixed time slice $t=\mathrm{const}$}, we cannot reconstruct the orginal field $\psi$ from the proposed evolution variables $\tilde{\psi}_t$, $\tilde{\psi}_r$.
This shows that the full equation \eqref{eq:main} cannot be written in terms of the variables $\tilde{\psi}_t$, $\tilde{\psi}_r$.
Another way to put this is to note that the prospective local energy ``norm'' derived from the quantity \eqref{eq:hienwrong} is in fact only a seminorm on the space of functions we are interested in which is too weak for our purposes.
Consequently, we dismiss the choice $k=1$ and follow the second possible path.

Choosing $k=-2$ we obtain the equation
$$ \tilde{\psi}_{tt}-\tilde{\psi}_{rr}+\frac{2}{r}\tilde{\psi}_r=0. $$
Note carefully that this is not a radial wave equation since the first order term has the wrong sign!
Differentiating with respect to $r$ yields
$$ \tilde{\psi}_{rtt}-\tilde{\psi}_{rrr}+\frac{2}{r}\tilde{\psi}_{rr}-\frac{2}{r^2}\tilde{\psi}_r=0. $$
In order to obtain a conserved quantity (``energy'') for this equation, one notes that 
$\tilde{\psi}_{rrr}-\frac{2}{r}\tilde{\psi}_{rr}=r^2 \partial_r (r^{-2}\tilde{\psi}_{rr})$, multiplies by $\frac{1}{r^2}\tilde{\psi}_{rt}$ and integrates by parts.
However, the resulting expression is not positive definite due to the wrong sign of the first--order term.
Thus, we make another transformation $\tilde{\psi}_r(t,r)=:r^k \hat{\psi}(t,r)$ to obtain
$$ \hat{\psi}_{tt}-\hat{\psi}_{rr}-2\frac{k-1}{r}\hat{\psi}_r-\frac{(k-1)(k-2)}{r^2}\hat{\psi}=0 $$
and by choosing $k=1$ we end up with the one--dimensional wave equation
$$ \hat{\psi}_{tt}-\hat{\psi}_{rr}=0 $$ 
on the half--line $r \geq 0$.
In terms of the original field $\psi$ we have
$$ \hat{\psi}(t,r)=r\psi_r(t,r)+2\psi(t,r) $$
and this implies the boundary condition $\hat{\psi}(t,0)=0$ for all $t$.
As a consequence, we obtain the conserved quantity
\begin{equation}
 \label{eq:enhi}
\frac{1}{2}\int_0^\infty \left (\hat{\psi}_t(t,r)^2+\hat{\psi}_r(t,r)^2 \right )dr 
\end{equation}
and this suggests to use $\hat{\psi}_{t}$ and $\hat{\psi}_r$ as evolution variables.
This is indeed possible now since the boundary condition $\hat{\psi}(t,0)=0$ allows us to express $\hat{\psi}$ in terms of $\hat{\psi}_r$ as
$$ \hat{\psi}(t,r)=\int_0^r \hat{\psi}_r(t,r')dr' $$
and as a consequence, for the original field $\psi$, we obtain
$$ \psi(t,r)=\frac{1}{r^2}\int_0^r r' \hat{\psi}(t,r')dr'=\frac{1}{r^2}\int_0^r r' \int_0^{r'} \hat{\psi}_r(t,s)ds dr'. $$
Now observe further that for the self--similar solution 
$$ \hat{\psi}^T(t,r)=r \psi^T_r(t,r)+2 \psi^T(t,r)=\frac{2\rho}{1+\rho^2}+4 \arctan \rho $$
we have
$$ \hat{\psi}^T_r(t,r)=(T-t)^{-1} \frac{2(3+\rho^2)}{(1+\rho^2)^2}=:(T-t)^{-1}g(\rho) $$
and
$$ \hat{\psi}^T_t(t,r)=(T-t)^{-1}\rho g(\rho) $$
where $\rho=\frac{r}{T-t}$.
This implies
\begin{align*} 
 \int_0^{T-t} \left ( \hat{\psi}^T_t(t,r)^2+\hat{\psi}^T_r(t,r)^2 \right )dr &=
(T-t)^{-2} \int_0^{T-t} \left [ \left (\tfrac{r}{T-t}\right )^2 g\left (\tfrac{r}{T-t} \right )^2+
g\left (\tfrac{r}{T-t} \right )^2 \right ] dr \\
&=(T-t)^{-1} \int_0^1 \left (\rho^2 g(\rho)^2+g(\rho)^2 \right )d\rho \\
&\simeq (T-t)^{-1}
\end{align*}
as $t \to T-$ and thus, the local ``energy'' norm derived from the conserved quantity \eqref{eq:enhi} blows up. Therefore, it is a promising candidate for a suitable norm to study the stability of $\psi^T$.
In this paper we develop the linearized perturbation theory about $\psi^T$ in the topology given by this higher energy norm.

\section{Mode stability}
\label{sec:mode}

\subsection{Similarity coordinates}
In this section we study so--called \emph{mode solutions} of the linearized equation around the fundamental self--similar solution $\psi^T$. We linearize Eq.~\eqref{eq:main} by inserting the ansatz $\psi=\psi^T+\varphi$ and neglecting terms of order $\varphi^2$ or higher. 
This yields the \emph{linearized equation}
\begin{equation}
\label{eq:mainrellin}
\varphi_{tt}-\varphi_{rr}-\frac{2}{r}\varphi_r+\frac{2\cos(2\psi^T)}{r^2}\varphi=0 \mbox{ in } \mc{C}_T
\end{equation}
with appropriate initial data.
The concrete form of the data is irrelevant for the mode analysis and therefore, we omit them in this section.
Since the solution $\psi^T$ is self--similar, it is natural to switch to similarity coordinates $(\tau,\rho)$ defined by $\tau:=-\log(T-t)$, $\rho:=\frac{r}{T-t}$.
\begin{figure}[h]
 \includegraphics{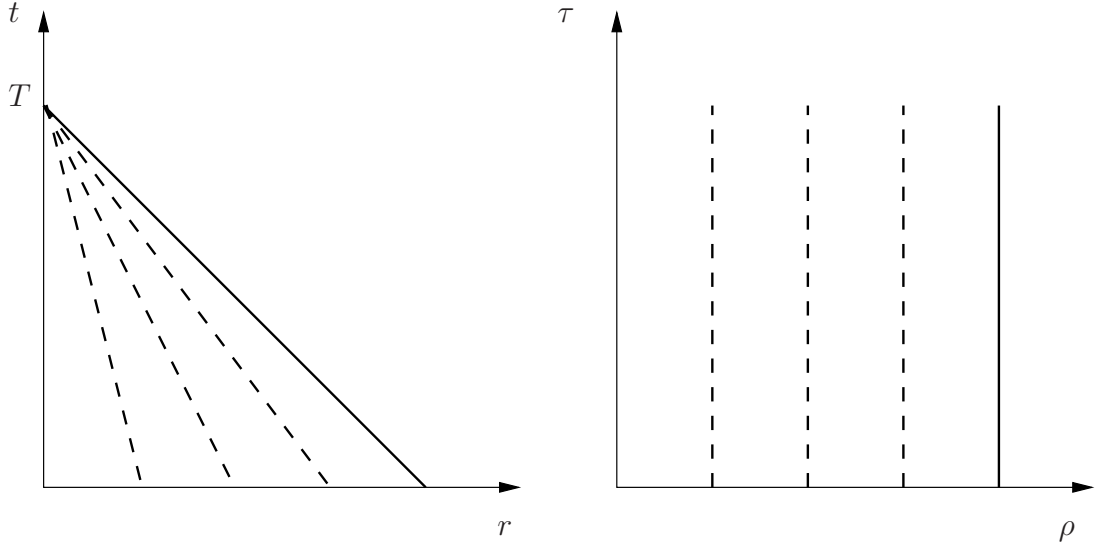}
 \caption{The backward lightcone $\mc{C}_T$ in $(t,r)$ and $(\tau,\rho)$ coordinates. We have also included some lines of constant $\rho$ (dashed).}
  \label{fig:css}
\end{figure}
Observe that this change of variables transforms the cone $\mc{C}_T$ to the infinite cylinder $$ \mc{Z}_T:=\{(\tau,\rho): \tau>-\log T, \rho \in [0,1]\}, $$ see Fig.~\ref{fig:css}.
In particular, the blow up of $\psi^T$ takes place as $\tau$ goes to infinity.
According to the chain rule, the derivatives transform as $\partial_t=e^{\tau}(\partial_\tau+\rho \partial_\rho)$ and $\partial_r=e^\tau \partial_\rho$.
In similarity coordinates $(\tau,\rho)$,
the evolution equation \eqref{eq:mainrellin} reads
\begin{equation}
\label{eq:linearcssscalar}
\phi_{\tau \tau}+\phi_\tau+2\rho \phi_{\tau \rho}-(1-\rho^2)\phi_{\rho \rho}-2\frac{1-\rho^2}{\rho}\phi_\rho+\frac{V(\rho)}{\rho^2}\phi=0 
\end{equation}
where $\phi(\tau,\rho):=\varphi(T-e^{-\tau}, e^{-\tau}\rho)$
and the ``potential'' $V$ is given by
\begin{equation} 
\label{eq:V}
V(\rho)=2 \cos(4 \arctan(\rho))=\frac{2(1-6 \rho^2+\rho^4)}{(1+\rho^2)^2}. 
\end{equation}
The first step in a stability analysis of such an equation is to look for \emph{mode solutions}, i.e., solutions of the form $\phi(\tau,\rho; \lambda)=e^{\lambda \tau}u(\rho,\lambda)$ for a complex number $\lambda$ and a sufficiently regular function $u(\cdot,\lambda)$.
The nonexistence of mode solutions with $\mathrm{Re}\lambda \geq 0$ is a necessary (though, in general, not sufficient) requirement for linear stability of the fundamental self--similar solution $\psi^T$.

\subsection{The eigenvalue equation}
In order to decide whether there are unstable mode solutions or not, we insert the ansatz $\phi(\tau,\rho; \lambda)=e^{\lambda \tau}u(\rho,\lambda)$ into Eq.~\eqref{eq:linearcssscalar} to derive the ordinary differential equation
\begin{equation}
\label{eq:evode}
-(1-\rho^2)u''(\rho,\lambda)-2\frac{1-\rho^2}{\rho}u'(\rho,\lambda)+2\lambda \rho u'(\rho,\lambda)+\frac{V(\rho)}{\rho^2}u(\rho,\lambda)+\lambda (\lambda+1)u(\rho,\lambda)=0
\end{equation}
which constitutes a generalized eigenvalue problem for the function $u(\cdot,\lambda)$ where $'$ always means derivative with respect to $\rho$.
Unfortunately, due to the complexity of the potential $V$, this equation cannot be solved explicitly and it seems to be a hard problem to exclude unstable mode solutions rigorously. 
Indeed, there are only partial results available, see below.
The rest of this section is devoted to the analysis of Eq.~\eqref{eq:evode}.

We start off with some basic observations. Eq.~\eqref{eq:evode} has two regular singular points at $\rho=0$ and $\rho=1$ (which is the location of the lightcone).
Furthermore, a brief inspection of the potential $V$ shows that additionally there are regular singular points outside the ``physical'' domain at $\rho=-1, \pm i$ as well as $\rho=\infty$.
By the change of variables $\rho \mapsto \rho^2$ one can reduce the number of regular singular points from six to four which shows that the general solution is given in terms of Heun's functions.
This, however, is not of much use for us since there is no rigorous theory of Heun's functions and most of the results are based on some sort of numerics.
Nevertheless, we can at least get information on the asymptotic behavior of the solutions around the singular points $\rho=0$ and $\rho=1$ by applying Frobenius' method (see e.g., \cite{miller}, \S 6.2).
The indices at $\rho=0$ are $\{1,-2\}$ and at $\rho=1$ we have $\{0,1-\lambda\}$ which shows that analytic solutions behave like $\rho$ around $0$ and like $1$ at $\rho=1$.
However, note carefully that the nonanalytic solution around $\rho=1$ behaves like $(1-\rho)^{1-\lambda}$ and it therefore becomes more and more regular as $\mathrm{Re}\lambda$ decreases.
It is not clear at this stage (and in fact it is false) that only analytic solutions of Eq.~\eqref{eq:evode} are relevant.
The precise regularity requirements can only be derived once one has a well--posed initial value formulation of Eq.~\eqref{eq:linearcssscalar} and a well--defined spectral problem.
This will be postponed to Sec.~\ref{sec:wp} and for the moment we only look for analytic solutions of Eq.~\eqref{eq:evode}. 
Thus, we say that $\lambda \in \mathbb{C}$ is an eigenvalue iff Eq.~\eqref{eq:evode} has a nontrivial analytic solution $u(\rho,\lambda)$.
Note also that it is sufficient to consider smooth solutions (i.e., solutions in $C^\infty[0,1]$) since any smooth solution of Eq.~\eqref{eq:evode} is automatically analytic by Frobenius' method.

\subsection{The gauge mode}
By a direct computation one immediately checks that the function $g(\rho)=\frac{\rho}{1+\rho^2}$ is a solution of Eq.~\eqref{eq:evode} with $\lambda=1$ and thus, there exists an unstable mode.
However, it is well--known that this instability emerges from the fact that $\psi^T$ is a family of solutions (depending on the free parameter $T$) rather than a single one.
Hence, the time translation symmetry of the original wave map equation \eqref{eq:main} is reflected by this unstable mode of the linearized operator and therefore, the mode $g$ is sometimes called a \emph{gauge mode}.
As a consequence, we really want to study stability ``modulo'' this gauge instability.
We will make this precise in Sec.~\ref{sec:wp}. 
For the moment we just ignore this gauge mode and look for other unstable modes.
Note also that there are no other ``gauge modes'' since all the remaining symmetries of the equation are eliminated by our restriction to co--rotational wave maps (of course, there is still the scaling symmetry but, when acting on a self--similar solution, scaling and time 
translation are equivalent).

\subsection{Known results}
The eigenvalue equation \eqref{eq:evode} has been extensively studied by several authors. First, there are a number of numerical works where essentially three different methods have been applied to show that there are no unstable modes \footnote{We always mean unstable modes apart from the gauge mode.}.
Probably the easiest thing to do is to employ a so--called shooting and matching technique where one integrates the equation \eqref{eq:evode} away from the singular points by using a standard numerical method like Runge--Kutta or variants thereof. One then tries to match the two solutions at $\rho=\frac{1}{2}$. Such a numerical analysis has been performed in \cite{ichdipl}, \cite{ich0}, \cite{bizon99}, \cite{liebling}. Special care has to be taken since one has to solve singular initial value problems but this can be done in a satisfactory way. A more serious disadvantage of this method is that one has to have good initial guesses and therefore, it is in principle possible to ``miss'' eigenvalues.
That is why one would like to have additional (theoretical) arguments that there are no eigenvalues far away from the real axis. In fact, we will prove such a statement, see Lemma \ref{lem:faraway} below.
Second, one can numerically integrate the time evolution equation \eqref{eq:linearcssscalar} 
and remove the gauge instability by a suitable projection \cite{ichdipl}, \cite{ich0}.
With this approach one cannot ``miss'' eigenvalues with large imaginary parts but it is numerically more delicate since one has to solve a PDE instead of an ODE.
Finally, it is possible to rewrite the problem \eqref{eq:evode} as a difference equation and use continued fractions expansions which eventually leads to a root finding problem for a transcendental function.
This is a highly accurate method and it has been successfully applied to Eq.~\eqref{eq:evode} in \cite{bizon05}.
The outcome of all these numerical studies is that the eigenvalue with the largest real part (apart from the gauge mode) is $\lambda \approx -0.542466$ and this strongly suggests the mode stability of the solution $\psi^T$.

As far as rigorous arguments are concerned, there are only partial results.
In \cite{ichpca2}, Eq.~\eqref{eq:evode} is transformed into a self--adjoint Sturm--Liouville form and this approach can be used to show that there are no unstable eigenvalues $\lambda$ with $\mathrm{Re}\lambda>1$.
Furthermore, by ODE techniques it has been proved that there are no \emph{real} unstable eigenvalues except for $\lambda=1$ \cite{ichpca1}.
However, there is no result so far that excludes unstable eigenvalues $\lambda$ with $0 \leq \mathrm{Re}\lambda <1$ and $\mathrm{Im}\lambda\not= 0$.
In what follows we push the boundary a bit further and prove a result that excludes eigenvalues $\lambda$ with $\mathrm{Re}\lambda \geq \frac{1}{2}$.

\subsection{The dual eigenvalue problem}
In \cite{ichpca2} there is given a factorization of the linear operator which can be used to obtain a ``dual'' eigenvalue problem which gives important new insights into the problem of mode stability.

\begin{lemma}
 \label{lem:dual}
Let $u(\cdot,\lambda) \in C^\infty[0,1]$ be a nontrivial solution of Eq.~\eqref{eq:evode}. Then either $\lambda=1$ and $u(\rho,1)=\frac{c\rho}{1+\rho^2}$ for a constant $c \in \mathbb{C}$ (i.e., $u(\cdot,\lambda)$ is a multiple of the gauge mode) or the equation
\begin{equation}
\label{eq:evodedual} -(1-\rho^2)v''(\rho,\lambda)-2\frac{1-\rho^2}{\rho}v'(\rho,\lambda)+2\lambda \rho v'(\rho,\lambda)+\frac{\tilde{V}(\rho)}{\rho^2}v(\rho,\lambda)+\lambda (\lambda+1)v(\rho,\lambda)=0, 
\end{equation} 
with the potential
$$ \tilde{V}(\rho):=\frac{6-2 \rho^2}{1+\rho^2}$$
has a nontrivial solution $v(\cdot,\lambda) \in C^\infty[0,1]$.
\end{lemma}

\begin{proof}
 Let $u(\cdot,\lambda) \in C^\infty[0,1]$ be a nontrivial 
solution of Eq.~\eqref{eq:evode}.
We define a new variable $\tilde{u}(\rho,\lambda):=(1-\rho^2)^{\lambda/2}u(\rho,\lambda)$.  Then $\tilde{u}(\cdot,\lambda)$ satisfies
$$ -(1-\rho^2)^2 \tilde{u}''(\rho,\lambda)-2\frac{(1-\rho^2)^2}{\rho}\tilde{u}'(\rho,\lambda)+\frac{(1-\rho^2)V(\rho)-\rho^2}{\rho^2}\tilde{u}(\rho,\lambda)
=-(1-\lambda)^2\tilde{u}(\rho,\lambda) $$
and this equation can be written as
\begin{equation}
\label{eq:BB}
\hat{\beta}\beta\tilde{u}(\cdot,\lambda)=-(1-\lambda)^2 \tilde{u}(\cdot,\lambda) 
\end{equation}
where the (formal) differential operators $\beta$ and $\hat{\beta}$ are given by
$$ \beta:=(1-\rho^2)\frac{d}{d\rho}-\frac{1-3\rho^2}{\rho(1+\rho^2)} $$
and
$$ \hat{\beta}:=-(1-\rho^2)\frac{d}{d\rho}-\frac{3-\rho^2}{\rho(1+\rho^2)}, $$
see \cite{ichpca2}.
The solutions of the first--order ODE $\beta \tilde{u}=0$ can be given explicitly and the one--dimensional solution space is spanned by the function $\rho \mapsto \frac{\rho \sqrt{1-\rho^2}}{1+\rho^2}$.
Thus, if $\beta \tilde{u}(\cdot,\lambda)=0$ then Eq.~\eqref{eq:BB} implies $\lambda=1$ and $u(\cdot,\lambda)$ is a multiple of the gauge mode.
Suppose $\beta\tilde{u}(\cdot,\lambda)\not=0$. 
Applying the operator $\beta$ to Eq.~\eqref{eq:BB} yields
\begin{equation}
\label{eq:BBsusy}
\beta\hat{\beta}\beta\tilde{u}(\cdot,\lambda)=-(1-\lambda)^2 \beta\tilde{u}(\cdot,\lambda) 
\end{equation}
and this is a second order ODE for the nonzero function $\tilde{v}(\cdot,\lambda):=\beta\tilde{u}(\cdot,\lambda)$.
The operator $\beta$ is singular at $\rho=0$ but due to the asymptotics $\tilde{u}(\rho,\lambda)\simeq \rho$ as $\rho \to 0+$ (Frobenius' method) we still have $\tilde{v}(\cdot,\lambda)\in C^\infty[0,1)$.
Furthermore, as $\rho \to 1-$, $\tilde{v}(\rho,\lambda)$ behaves like $(1-\rho)^{\lambda/2}$. 
Explicitly, Eq.~\eqref{eq:BBsusy} reads
$$ -(1-\rho^2)^2 \tilde{v}''(\rho,\lambda)-2\frac{(1-\rho^2)^2}{\rho}\tilde{v}'(\rho,\lambda)+\frac{(1-\rho^2)\tilde{V}(\rho)-\rho^2}{\rho^2}\tilde{v}(\rho,\lambda)
=-(1-\lambda)^2\tilde{v}(\rho,\lambda) $$
with the new potential $\tilde{V}(\rho)=\frac{6-2\rho^2}{1+\rho^2}$.
Setting $v(\rho,\lambda):=(1-\rho^2)^{-\lambda/2}\tilde{v}(\rho,\lambda)$ we obtain $v(\cdot,\lambda)\in C^\infty[0,1]$ and $v(\cdot,\lambda)$ satisfies Eq.~\eqref{eq:evodedual}. 
\end{proof}

\subsection{Nonexistence of eigenvalues with $\mathrm{Re}\lambda \geq \frac{1}{2}$}

Based on Lemma \ref{lem:dual} we can now prove that the gauge mode is the only eigenmode with $\mathrm{Re}\lambda \geq \frac{1}{2}$. 

\begin{proposition}
\label{prop:modstab}
Let $u(\cdot,\lambda) \in C^\infty[0,1]$ be a nontrivial solution of Eq.~\eqref{eq:evode} with $\mathrm{Re}\lambda \geq \frac{1}{2}$. Then $\lambda=1$ and there exists a (nonzero) constant $c \in \mathbb{C}$ such that $u(\rho,1)=\frac{c\rho}{1+\rho^2}$. 
\end{proposition} 

\begin{proof}
Let $u(\cdot,\lambda) \in C^\infty[0,1]$ be a nonzero solution of Eq.~\eqref{eq:evode} with $\mathrm{Re}\geq \frac{1}{2}$ and assume that $\lambda \not=1$.
According to Lemma \ref{lem:dual} there exists a nonzero $v(\cdot,\lambda) \in C^\infty[0,1]$ such that
$$ -(1-\rho^2)v''(\rho,\lambda)-2\frac{1-\rho^2}{\rho}v'(\rho,\lambda)+2\lambda \rho v'(\rho,\lambda)+\frac{\tilde{V}(\rho)}{\rho^2}v(\rho,\lambda)+\lambda (\lambda+1)v(\rho,\lambda)=0.$$
Thus, $\chi(\tau,\rho;\lambda):=e^{\lambda \tau} v(\rho,\lambda)$ is a smooth solution of the evolution equation
$$\chi_{\tau \tau}+\chi_\tau+2\rho \chi_{\tau \rho}-(1-\rho^2)\chi_{\rho \rho}-2\frac{1-\rho^2}{\rho}\chi_\rho+\frac{\tilde{V}(\rho)}{\rho^2}\chi=0 $$
and transforming back to the physical coordinates $(t,r)$ this simply reads
$$ \chi_{tt}-\chi_{rr}-\frac{2}{r}\chi_r+\frac{\tilde{V}(\frac{r}{T-t})}{r^2}\chi=0. $$
Motivated by the energy of this equation, consider the positive smooth function
$$ F_\chi(t):=\frac{1}{2}\int_0^{T-t}\left [r^2 |\chi_t(t,r;\lambda)|^2+r^2|\chi_r(t,r;\lambda)|^2
+\tilde{V}\left(\frac{r}{T-t}\right )|\chi(t,r;\lambda)|^2 \right ]dr. $$
Differentiation yields
\begin{align*}
F_\chi'(t)&= \mathrm{Re}\int_0^{T-t} \left [r^2 \chi_{tt} \overline{\chi_t}+r^2 \chi_r \overline{\chi_{rt}}+
\tilde{V}\left(\frac{r}{T-t}\right )\chi \overline{\chi_t}+
\frac{1}{2}
\underbrace{\tilde{V}'\left(\frac{r}{T-t}\right )}_{\leq 0}\frac{r}{(T-t)^2}|\chi|^2 dr \right ]\\
&\quad -\frac{1}{2}(T-t)^2 |\chi_t(t,T-t;\lambda)|^2-\frac{1}{2}(T-t)^2 |\chi_r(t,T-t;\lambda)|^2-\frac{1}{2}\tilde{V}(1)|\chi(t,T-t;\lambda)|^2 \\
&\leq  \mathrm{Re}\int_0^{T-t} \left [\partial_r(r^2 \chi_r)\overline{\chi_t}
+r^2 \chi_r \overline{\chi_{rt}} \right ] dr \\
&\quad -\frac{1}{2}(T-t)^2 |\chi_t(t,T-t;\lambda)|^2-\frac{1}{2}(T-t)^2 |\chi_r(t,T-t;\lambda)|^2-\frac{1}{2}\tilde{V}(1)|\chi(t,T-t;\lambda)|^2 \\
& \leq -|\chi(t,T-t;\lambda)|^2
\end{align*}
by integration by parts and Cauchy's inequality.
Inserting $\chi(t,T-t;\lambda)=(T-t)^{-\lambda}v(1)$ and integrating from $0$ to $t$ for $t<T$ yields
$$ F_\chi(t)\leq -|v(1)|^2 \int_0^t (T-t')^{-2\mathrm{Re}\lambda}dt'+F_\chi(0). $$
However, from the asymptotic behavior of $v$ we know that $v(1)\not=0$ and therefore we have a contradiction since $F_\chi$ is positive but the right--hand side of the above inequality goes to $-\infty$ as $t \to T-$ (recall that we assume $\mathrm{Re}\lambda \geq \frac{1}{2}$).
\end{proof}

\subsection{Reformulation as a resonance problem}
We conclude this section by showing that the mode stability problem is in fact equivalent to finding resonances of the $\ell=2$ Schr\"odinger operator with a potential that has inverse square decay towards infinity.
We formulate the problem for the ``dual'' eigenvalue equation
$$
-(1-\rho^2)v''(\rho,\lambda)-2\frac{1-\rho^2}{\rho}v'(\rho,\lambda)+2\lambda \rho v'(\rho,\lambda)+\frac{\tilde{V}(\rho)}{\rho^2}v(\rho,\lambda)+\lambda (\lambda+1)v(\rho,\lambda)=0 
$$
with potential $\tilde{V}(\rho)=\frac{6-2\rho^2}{1+\rho^2}$.
We are interested in smooth solutions and according to numerics we expect the eigenvalue with the largest real part to be $\lambda \approx -0.54$.
Now suppose $v(\cdot,\lambda)$ is an eigenfunction, i.e., a smooth solution to the above equation.
By Frobenius' method we have the asymptotic behavior $v(\rho,\lambda)\simeq \rho^2$ as $\rho \to 0+$ and $v(\rho,\lambda)\simeq 1$ as $\rho \to 1-$.
Defining a new variable $\tilde{v}$ by 
$$ \tilde{v}(x,\mu):=(1-\tanh^2 x)^{\lambda/2}\sinh x \:u(\tanh x,1-i\mu),$$ 
the eigenvalue problem transforms into
\begin{equation}
\label{eq:SG}
-\tilde{v}''(x,\lambda)+\frac{6}{x^2}\tilde{v}(x,\lambda)+Q(x)\tilde{v}(x,\lambda)=\mu^2 \tilde{v}(x,\lambda) 
\end{equation}
on the half--line $x \geq 0$ with the potential
$$ Q(x)=\frac{\tilde{V}(\tanh x)}{\sinh^2 x}-\frac{6}{x^2}. $$
The outgoing Jost solution $f_+(\cdot,\mu)$ is defined as the solution of Eq.~\eqref{eq:SG} that behaves like $f_+(x,\mu)\sim e^{i\mu x}$ as $x \to \infty$.
Furthermore, we denote by $f_0(\cdot,\mu)$ the solution of Eq.~\eqref{eq:SG} which is square--integrable near 
$x = 0$.
The eigenfunction $\tilde{v}(\cdot,\mu)$ has the asymptotic behavior $\tilde{v}(x,\mu) \simeq x^3$ as $x \to 0+$ and $\tilde{v}(x,\mu)\simeq e^{i\mu x}$ as $x \to \infty$.
Thus, finding eigenvalues amounts to matching the outgoing Jost solution $f_+(\cdot,\mu)$ to $f_0(\cdot,\mu)$ or, in other words, $\lambda$ is an eigenvalue iff $\mu=i(\lambda-1)$ 
is a zero of the Wronskian 
$$ W(f_0,f_+)(\mu)=f_0(\cdot,\mu)f_+'(\cdot,\mu)-f_0'(\cdot,\mu)f_+(\cdot,\mu). $$ 
We are mainly interested in eigenvalues $\lambda$ with $\mathrm{Re}\lambda<1$ and this translates into the condition $\mathrm{Im}\mu<0$ since $\mu=-\mathrm{Im}\lambda+i\mathrm{Re}(\lambda-1)$.
However, zeros of the Wronskian $W(f_0,f_+)(\mu)$ with $\mathrm{Im}\mu<0$ are precisely the resonances or scattering poles of the Schr\"odinger operator defined by Eq.~\eqref{eq:SG}.

\section{The linear perturbation theory}
\label{sec:wp}

\subsection{First--order formulation}
We intend to rewrite the wave maps equation \eqref{eq:mainreg} as a system which is first--order in time.
From now on we are a bit more careful and specify the full Cauchy problem we want to study, including the initial data and the domain.
Recall the definition of the backward lightcone 
$$ \mc{C}_T:=\{(t,r): t \in (0,T), r \in [0,T-t]\}. $$
Precisely, we are interested in the evolution problem
\begin{equation}
 \label{eq:maincauchy}
\left \{ \begin{array}{l}
\psi_{tt}(t,r)-\psi_{rr}(t,r)-\frac{2}{r}\psi_r(t,r)+\frac{\sin(2\psi(t,r))}{r^2}=0 \mbox{ for } (t,r) \in \mc{C}_T \\
\psi(0,r)=f(r), \psi_t(0,r)=g(r) \mbox{ for }r \in [0,T]
         \end{array} \right .
\end{equation}
with given initial data $(f,g)$.
For the remainder of this section we assume that a solution $\psi$ of the Cauchy problem \eqref{eq:maincauchy} exists and it is sufficiently regular so that all of the following computations are justified.
In particular, we assume that $\psi$ satisfies the regularity condition $\psi(t,0)=0$ which implies $f(0)=g(0)=0$.
Actually, we want to study the flow about the self--similar solution $\psi^T$ and thus, we are interested in ``small'' perturbations $\varphi$ of $\psi^T$.
Thus, it is convenient to rewrite the problem relative to $\psi^T$ by setting $\psi=\psi^T+\varphi$ and expanding the nonlinearity in Eq.~\eqref{eq:maincauchy} as
$$ \sin(2(\psi^T+\varphi))=\sin(2\psi^T)+2\cos(2\psi^T)\varphi+N_T(\varphi) $$
where $N_T(x)=O(x^2)$ for $x \to 0$ is the nonlinear remainder.
The Cauchy problem \eqref{eq:maincauchy} becomes
\begin{equation}
\label{eq:mainrel}
\left \{ \begin{array}{l}
\varphi_{tt}-\varphi_{rr}-\frac{2}{r}\varphi_r+\frac{2}{r^2}\varphi+\frac{2\cos(2\psi^T)-2}{r^2}\varphi+\frac{N_T(\varphi)}{r^2}=0 \mbox{ in } \mc{C}_T \\
\varphi(0,r)=f(r)-\psi^T(0,r), \varphi_t(0,r)=g(r)-\psi^T_t(0,r) \mbox{ for }r \in [0,T]
         \end{array} \right .
\end{equation}
where we have suppressed the arguments of $\varphi$ and $\psi^T$ so as not to overload the equation with too many symbols.
From the regularity requirement $\psi(t,0)=0$ we obtain the boundary condition $\varphi(t,0)=0$ for all $t \in [0,T)$.
Next, motivated by the discussion in Sec.~\ref{sec:bunorm}, we introduce the new variable $\tilde{\varphi}(t,r):=r^2 \varphi(t,r)$ which yields
\begin{equation}
\left \{ \begin{array}{l}
\tilde{\varphi}_{tt}-\tilde{\varphi}_{rr}+\frac{2}{r}\tilde{\varphi}_r
+\frac{2\cos(2\psi^T)-2}{r^2}\tilde{\varphi}+N_T(r^{-2}\tilde{\varphi})=0 \mbox{ in } \mc{C}_T  \\
\tilde{\varphi}(0,r)=r^2 [f(r)-\psi^T(0,r)], \tilde{\varphi}_t(0,r)=r^2 [g(r)- \psi^T_t(0,r)] \mbox{ for }r \in [0,T]
         \end{array} \right . 
\end{equation}
and again, we occasionally omit the arguments of $\tilde{\varphi}$ and $\psi^T$.
In order to write this Cauchy problem as a first--order system in time, we introduce the variables $$ \varphi_1(t,r):=\frac{\tilde{\varphi}_t(t,r)}{T-t},\quad \varphi_2(t,r):=\frac{\tilde{\varphi}_r(t,r)}{r}. $$
This choice is motivated by the fact that
$\varphi_j$, $j=1,2$, scale like the original field $\varphi$ or, in other words, they have the same physical dimension.
Furthermore,
the higher energy \eqref{eq:enhi} is a simple expression in terms of $\varphi_j$.
Note also that the field $\tilde{\varphi}$ can be reconstructed from $\varphi_2$ by
$$ \tilde{\varphi}(t,r)=\int_0^r r' \varphi_2(t,r')dr' $$
thanks to the boundary condition $\tilde{\varphi}(t,0)=0$ for all $t \in [0,T)$.
We obtain the system
\begin{equation}
\label{eq:main1st}
\left \{ \begin{array}{l}
\left .\begin{array}{l}
\partial_t \varphi_1=(T-t)^{-1}\left [\varphi_1+r\partial_r \varphi_2
-\varphi_2-\frac{2\cos(2\psi^T)-2}{r^2}\int r\varphi_2-N_T \left ( r^{-2}\int r \varphi_2 \right ) \right ]
 \\
\partial_t \varphi_2=\frac{T-t}{r}\partial_r \varphi_1 \end{array} \right \}
\mbox{ in } \mc{C}_T \\
\left. \begin{array}{l}
\varphi_1(0,r)=\frac{1}{T}r^2\left [g(r)-\psi^T_t(0,r) \right ] \\
\varphi_2(0,r)=r^{-1}\partial_r r^2 \left [f(r)-\psi^T(0,r) \right ] 
\end{array} \right \}
\mbox{ for } r \in [0,T]
\end{array} \right .
\end{equation} 
where $\int r \varphi_2$ is shorthand for $\int_0^r r' \varphi_2(t,r')dr'$.
Now we switch to similarity coordinates 
$$ \tau=-\log (T-t),\quad \rho=\tfrac{r}{T-t}. $$
Then we have
$$ \int_0^r r'\varphi_2(t,r')dr'=e^{-2\tau}\int_0^\rho \rho' \varphi_2(t, e^{-\tau}\rho')d\rho' $$
and $t=0$ corresponds to $\tau=-\log T$.
Thus, the system \eqref{eq:main1st} transforms into
\begin{equation}
\label{eq:main1stcss}
\left \{ \begin{array}{l}
\left. \begin{array}{l}
\partial_\tau \phi_1=-\rho \partial_\rho \phi_1+\phi_1+\rho \partial_\rho \phi_2
-\phi_2-\frac{2\cos(2 f_0)-2}{\rho^2}\int \rho \phi_2-N_T \left ( \rho^{-2}\int \rho \phi_2 \right ) 
 \\
\partial_\tau \phi_2=\frac{1}{\rho}\partial_\rho \phi_1-\rho \partial_\rho \phi_2 \end{array} \right \}
\mbox{ in }\mc{Z}_T \\
\left. \begin{array}{l}
\phi_1(-\log T,\rho)=T\rho^2\left [g(T\rho)-\psi^T_t(0,T\rho) \right ] \\
\phi_2(-\log T,\rho)=T\rho \left [f'(T\rho)-\psi^T_r(0,T\rho) \right ] +2 \left [f(T\rho)-\psi^T(0,T\rho)
\right ] 
\end{array} \right \}
\mbox{ for } \rho \in [0,1]
\end{array} \right .
\end{equation} 
where $\phi_j(\tau,\rho):=\varphi_j(T-e^{-\tau}, e^{-\tau}\rho)$, $j=1,2$ and $f_0(\rho)=2 \arctan(\rho)$ is the fundamental self--similar solution in similarity coordinates.
Tracing back the above transformations, we see that 
$(\phi_1, \phi_2)$ is a (sufficiently regular) solution of \eqref{eq:main1stcss} if and only if $\psi$ given by
 $$ \psi(t,r)=\psi^T(t,r)+\frac{1}{r^2}\int_0^r r' \phi_2\left (-\log(T-t), \tfrac{r'}{T-t} \right )dr' $$
 is a (sufficiently regular) solution of \eqref{eq:maincauchy}.
From now on we restrict ourselves to the linearized problem, i.e., we drop the nonlinearity in Eq.~\eqref{eq:main1stcss} and consider the system
\begin{equation}
\label{eq:main1stlinear}
\left \{ \begin{array}{l}
\left. \begin{array}{l}
\partial_\tau \phi_1=-\rho \partial_\rho \phi_1+\phi_1+\rho \partial_\rho \phi_2
-\phi_2-\frac{2\cos(2 f_0)-2}{\rho^2}\int \rho \phi_2 
 \\
\partial_\tau \phi_2=\frac{1}{\rho}\partial_\rho \phi_1-\rho \partial_\rho \phi_2 \end{array} \right \}
\mbox{ in }\mc{Z}_T \\
\left. \begin{array}{l}
\phi_1(-\log T,\rho)=T\rho^2\left [g(T\rho)-\psi^T_t(0,T\rho) \right ] \\
\phi_2(-\log T,\rho)=T\rho \left [f'(T\rho)-\psi^T_r(0,T\rho) \right ] +2 \left [f(T\rho)-\psi^T(0,T\rho)
\right ] 
\end{array} \right \}
\mbox{ for } \rho \in [0,1]
\end{array} \right .
\end{equation} 
Our aim is to study Eq.~\eqref{eq:main1stlinear} by means of semigroup theory, i.e., we rewrite the problem as an ordinary differential equation on a suitable Hilbert space.

\subsection{Function spaces and operator formulation}

We set
$$\tilde{H}_1:=\{u \in C^2[0,1]: u(0)=u'(0)=0\}, \quad
\tilde{H}_2:=\{u \in C^1[0,1]: u(0)=0\} $$
and define two sesquilinear forms $(\cdot|\cdot)_1$, $(\cdot|\cdot)_2$ by
$$ (u|v)_1:=\int_0^1 u'(\rho)\overline{v'(\rho)}\frac{d\rho}{\rho^2}, \quad 
(u|v)_2:=\int_0^1 u'(\rho)\overline{v'(\rho)}d\rho. $$

\begin{lemma}
 The sesquilinear form $(\cdot|\cdot)_j$ defines an inner product on $\tilde{H}_j$, $j=1,2$.
\end{lemma}

\begin{proof}
 Obviously, $(\cdot|\cdot)_2$ is well--defined on all of $\tilde{H}_2 \times \tilde{H}_2$. 
 Furthermore, for $u\in \tilde{H}_2$, $(u|u)_2=0$ is equivalent to $u'=0$ which, in turn, is equivalent to $u=0$ thanks to the boundary condition $u(0)=0$. This proves the assertion for $j=2$.
 For $j=1$ we apply Hardy's inequality to see that
 $$ (u|u)_1=\int_0^1 \frac{|u'(\rho)|^2}{\rho^2}d\rho \lesssim \int_0^1 |u''(\rho)|^2 d\rho < \infty $$
 for $u \in \tilde{H}_1$ which is justified by the boundary condition $u'(0)=0$.
 As a consequence, $(\cdot|\cdot)_1$ is well--defined on all of $\tilde{H}_1 \times \tilde{H}_1$.
The rest of the proof is identical to the case $j=2$ from above.
\end{proof}

As usual, the inner product on $\tilde{H}_j$ induces a norm $\|\cdot\|_j:=\sqrt{(\cdot|\cdot)_j}$, $j=1,2$.
Moreover, we denote by $(\cdot|\cdot)$ the canonically induced inner product on 
$\tilde{\mc{H}}:=\tilde{H}_1 \times \tilde{H}_2$, i.e., 
$$ (\mb{u}|\mb{v}):=(u_1|v_1)_1+(u_2|v_2)_2 $$
and $\|\cdot\|:=\sqrt{(\cdot|\cdot)}$.
The vector space $\tilde{\mc{H}}$ equipped with $(\cdot|\cdot)$ is a pre--Hilbert space and we denote by $\mc{H}=H_1 \times H_2$ its completion.
We have the following convenient density result.

\begin{lemma}
\label{lem:dense}
 The set $C_c^\infty (0,1]=\{u \in C^\infty [0,1]: \mathrm{supp}(u) \subset (0,1]\}$ is dense in $H_j$, $j=1,2$.
\end{lemma}

\begin{proof}
 Let $u \in \tilde{H}_1$. By definition, this implies $\frac{u'}{w} \in L^2(0,1)$ where $w(\rho):=\rho$.
 It is well--known that $C^\infty_c(0,1)$ is dense in $L^2(0,1)$ (see e.g., \cite{adams}, p.~31, Theorem 2.19) and thus, for any given $\varepsilon>0$, there exists a $v \in C^\infty_c(0,1)$ such that
$\|\frac{u'}{w}-v\|_{L^2(0,1)}<\varepsilon$.
Set 
$$ \tilde{u}(\rho):=\int_0^\rho \rho' v(\rho')d\rho'. $$
Then $\tilde{u} \in C^\infty_c(0,1]$ and $\frac{\tilde{u}'(\rho)}{\rho}=v(\rho)$.
Consequently, we obtain
$\|u-\tilde{u}\|_1=\|\frac{u'}{w}-v \|_{L^2(0,1)}<\varepsilon$ and, since
$\tilde{H}_1$ is dense in $H_1$ by construction, this implies the claim for $H_1$. The proof for $H_2$ is analogous.
\end{proof}

For later reference it is useful to note the following embedding and the fact that the boundary conditions for the functions ``survive'' the completion procedure.

\begin{lemma}
 \label{lem:HinC}
We have the embedding $\mc{H} \hookrightarrow C[0,1] \times C[0,1]$ and $\mb{u} \in \mc{H}$ implies $\mb{u}(0)=0$.
\end{lemma}

\begin{proof}
Note first that for $u \in \tilde{H}_2$ we have
 $$ |u(\rho)|= \left |\int_0^\rho u'(\rho)d\rho \right |\leq \int_0^1 |u'(\rho)|d\rho \leq \left (\int_0^1 |u'(\rho)|^2 d\rho \right )^{1/2}=\|u'\|_{L^2(0,1)}=\|u\|_2 $$
 for all $\rho \in [0,1]$ and this computation also implies $|u(\rho)|\leq \|u\|_1$ for all $u \in \tilde{H}_1$ and $\rho \in [0,1]$.
 By approximation, these estimates extend to $H_2$ and $H_1$, respectively, and we obtain
$$ \|\mb{u}\|_{L^\infty(0,1) \times L^\infty(0,1)}^2=\|u_1\|_{L^\infty(0,1)}^2+\|u_2\|_{L^\infty(0,1)}^2 \leq \|\mb{u}\|^2 $$ 
for all $\mb{u} \in \mc{H}$ by taking the supremum over $\rho \in [0,1]$.
Now let $\mb{u} \in \mc{H}$ be arbitrary. By construction, there exists a sequence $(\mb{u}_j) \subset \tilde{\mc{H}}$ such that $\|\mb{u}-\mb{u}_j\| \to 0$ as $j \to \infty$.
The above inequality implies
$$ \|\mb{u}-\mb{u}_j\|_{L^\infty(0,1)\times L^\infty(0,1)} \lesssim \|\mb{u}-\mb{u}_j\| \to 0 $$
as $j \to \infty$ and, 
since the uniform limit of continuous functions is continuous, this shows that
$\mb{u} \in C[0,1] \times C[0,1]$.
In particular, $\mb{u}_j \to \mb{u}$ pointwise and, since $\mb{u}_j(0)=0$ for all $j$, we obtain $\mb{u}(0)=0$.
\end{proof}

Now we define a linear operator $\tilde{L}_0: \mc{D}(\tilde{L}_0) \subset \mc{H} \to \mc{H}$ that represents the spatial differential operator in Eq.~\eqref{eq:main1stlinear} as follows.
We set 
$$ \mc{D}(\tilde{L}_0):=\{u \in C^3[0,1]: u(0)=u'(0)=u''(0)=0\} \times \{u \in C^3[0,1]: u(0)=0\} $$ and
$$ \tilde{L}_0 \mb{u}(\rho):=\left (\begin{array}{c}
-\rho u_1'(\rho)+u_1(\rho)+\rho u_2'(\rho)-u_2(\rho) \\
\frac{1}{\rho}u_1'(\rho)-\rho u_2'(\rho) 
                       \end{array} \right ). $$
It should be emphasized that the definition of $\mc{D}(\tilde{L}_0)$ 
is a crucial point in the whole construction. 
One has to be very careful with the boundary conditions so as not to ``lose'' solutions one is actually interested in.
To see that this is not the case, recall that, in terms of the original perturbation $\varphi$, the field $\varphi_1$ is given by
$$ \varphi_1(t,r)=\frac{\tilde{\varphi}_t(t,r)}{T-t}=\frac{r^2 \varphi_t(t,r)}{T-t} $$
and by regularity, we may assume that $\varphi_t(t,r)=O(r)$ as $r \to 0+$.
Since $\phi_1$ is nothing but $\varphi_1$ in similarity coordinates, we may assume $\phi_1(\tau,\rho)=O(\rho^3)$ as $\rho \to 0+$.
Analogously, one obtains $\phi_2(\tau,\rho)=O(\rho)$ as $\rho \to 0+$. 
This justifies the chosen boundary conditions.

\begin{lemma}
\label{lem:tildeL0}
For any $\mb{u} \in \mc{D}(\tilde{L}_0)$ we have $\tilde{L}_0 \mb{u} \in \tilde{\mc{H}}$.
\end{lemma}

\begin{proof}
 Let $\mb{u} \in \mc{D}(\tilde{L}_0)$. 
The first component $[\tilde{L}_0 \mb{u}]_1$ of $\tilde{L}_0 \mb{u}$ is obviously twice continuously differentiable and satisfies the boundary conditions $[\tilde{L}_0 \mb{u}]_1(0)=[\tilde{L}_0 \mb{u}]_1'(0)=0$.
 Thus, we have $[\tilde{L}_0 \mb{u}]_1 \in \tilde{H}_1$.
For the second component we obviously have $[\tilde{L}_0 \mb{u}]_2 \in C^1(0,1]$ and by 
applying de l'H\^ opital's rule we obtain
$$ \lim_{\rho \to 0+}[\tilde{L}_0 \mb{u}]_2(\rho)=\lim_{\rho \to 0+} \frac{u_1'(\rho)}{\rho}=\lim_{\rho \to 0+}u_1''(\rho)=0 $$ as well as
$$ \lim_{\rho \to 0+}[\tilde{L}_0 \mb{u}]_2'(\rho)=\lim_{\rho \to 0+}\frac{\rho u_1''(\rho)-u_1'(\rho)}{\rho^2}-u_2'(0)=\lim_{\rho \to 0+}\frac{u_1'''(\rho)}{2}-u_2'(0)=\frac{u_1'''(0)}{2}-u_2'(0) 
$$
which shows that $[\tilde{L}_0\mb{u}]_2 \in \tilde{H}_2$ as claimed.
\end{proof}

As a consequence of Lemma \ref{lem:tildeL0}, $\tilde{L}_0$ can be viewed as an unbounded operator 
$$  \tilde{L}_0: \mc{D}(\tilde{L}_0)\subset \mc{H} \to \mc{H} $$
and Lemma \ref{lem:dense} implies that $\tilde{L}_0$ is densely defined.
We also define another linear operator $L'$ that represents the ``potential term'' in Eq.~\eqref{eq:main1stlinear} by $\mc{D}(L')=\tilde{\mc{H}}$ and
$$ L' \mb{u}(\rho):=\left (\begin{array}{c}
-V_1(\rho)\int_0^\rho \rho' u_2(\rho')d\rho' \\ 0
                           \end{array} \right )
$$
where we have set
\begin{equation} 
\label{eq:V1}
V_1(\rho):=\frac{V(\rho)-2}{\rho^2}=\frac{2\cos(2f_0(\rho))-2}{\rho^2}=-\frac{16}{(1+\rho^2)^2},  
\end{equation}
cf.~Eq.~\eqref{eq:V}.
Obviously, we have $L'\mb{u} \in \tilde{\mc{H}}$ for any $\mb{u} \in \mc{D}(L')$ and thus, $L'$ may be viewed as a linear operator
$$ L': \mc{D}(L') \subset \mc{H} \to \mc{H}. $$ 

\begin{lemma}
\label{lem:L'bounded}
 The operator $L': \mc{D}(L')\subset \mc{H} \to \mc{H}$ is bounded. As a consequence, it uniquely extends to a bounded linear operator $L': \mc{H} \to \mc{H}$.
\end{lemma}

\begin{proof}
By Lemma \ref{lem:HinC} we have
 \begin{equation}
 \label{eq:proofL'bound}
\|u\|_{L^2(0,1)}^2=\int_0^1 |u(\rho)|^2 d\rho \leq \int_0^1 \|u\|_2^2 d\rho=\|u\|_2^2 
\end{equation}
for all $u \in H_2$.
Now let $\mb{u} \in \tilde{\mc{H}}$. Then we have
\begin{align*}
 \|L'\mb{u}\|^2&=\int_0^1 \left |V_1'(\rho)\int_0^\rho \rho' u_2(\rho')d\rho' +V_1(\rho)\rho u_2(\rho) \right |^2 \frac{d\rho}{\rho^2} \\
 &\lesssim \|V_1'\|_{L^\infty(0,1)}^2\int_0^1 \frac{1}{\rho^2}\left |\int_0^\rho \rho' u_2(\rho')d\rho' \right |^2 d\rho +\|V_1\|_{L^\infty(0,1)}^2 \int_0^1 |u_2(\rho)|^2 d\rho \\
 &\lesssim \int_0^1 |\rho u_2(\rho)|^2 d\rho+\|u_2\|_2^2 \lesssim \|u_2\|_2^2 \\
 &\lesssim \|\mb{u}\|^2
\end{align*}
by Hardy's inequality and the estimate \eqref{eq:proofL'bound}.
Since $\tilde{\mc{H}}$ is dense in $\mc{H}$, $L'$ can be extended to all of $\mc{H}$ by an elementary result of functional analysis.
\end{proof}

In what follows we assume that $L'$ is defined on all of $\mc{H}$.
A preliminary operator formulation of Eq.~\eqref{eq:main1stlinear} is given by
\begin{equation}
 \label{eq:linoppre}
 \left \{ \begin{array}{l}
\frac{d}{d\tau}\Phi(\tau)=(\tilde{L}_0+L')\Phi(\tau) \mbox{ for }\tau>-\log T \\
\Phi(-\log T)=\mb{u}_T
          \end{array} \right .
\end{equation}
where $\Phi: [-\log T, \infty) \to \mc{H}$ and 
$$ \mb{u}_T(\rho):=\left ( \begin{array}{c} 
T\rho^2\left [g(T\rho)-\psi^T_t(0,T\rho) \right ] \\
T\rho \left [f'(T\rho)-\psi^T_r(0,T\rho) \right ] +2 \left [f(T\rho)-\psi^T(0,T\rho)
\right ]
                          \end{array} \right ). $$

\subsection{Well--posedness of the linearized evolution}

For simplicity, we consider the free evolution
\begin{equation}
 \label{eq:freeoppre}
 \left \{ \begin{array}{l}
\frac{d}{d\tau}\Phi(\tau)=\tilde{L}_0\Phi(\tau) \mbox{ for }\tau>-\log T \\
\Phi(-\log T)=\mb{u}_T
          \end{array} \right .
\end{equation}
first.
Note that the formal solution to this problem is ``$\Phi(\tau)=\exp(\tau \tilde{L}_0)\mb{u}_T$'', 
however, it is not clear how to define the exponential in this context since the operator $\tilde{L}_0$ is unbounded.
We apply semigroup theory to assign a precise meaning to ``$\exp(\tau \tilde{L}_0)$''.

\begin{proposition}
\label{prop:gen0}
 The operator $\tilde{L}_0$ is closable and its closure $L_0$ generates a strongly continuous one--parameter semigroup $S_0: [0,\infty) \to \mc{B}(\mc{H})$ satisfying
 $$ \|S_0(\tau)\|\leq e^{-\frac{1}{2}\tau} $$
 for all $\tau \geq 0$.
 In particular, the Cauchy problem
\begin{equation*}
 \left \{ \begin{array}{l}
\frac{d}{d\tau}\Phi(\tau)=L_0\Phi(\tau) \mbox{ for }\tau>-\log T \\
\Phi(-\log T)=\mb{u}_T
          \end{array} \right .
\end{equation*}
for $\mb{u}_T \in \mc{D}(L_0)$ has a unique solution which is given by
$$ \Phi(\tau)=S_0(\tau+\log T)\mb{u}_T $$
for all $\tau \geq -\log T$.
\end{proposition}

\begin{proof}
Let $\mb{u} \in \mc{D}(\tilde{L}_0)$.
In the following calculation, all integrals run from $0$ to $1$ and for simplicity we omit the arguments of the functions.
Integrating by parts, we obtain
\begin{align*}
 \Re(\tilde{L}_0 \mb{u}|\mb{u})&=\Re \int \frac{1}{\rho^2} \left [ -\rho u_1'+u_1+\rho u_2'-u_2 \right ]'\overline{u_1'} +\Re \int \left [\frac{1}{\rho}u_1'-\rho u_2' \right ]'\overline{u_2'} \\
&=-\Re \int \frac{1}{\rho}u_1'' \overline{u_1'}+\Re \int \frac{1}{\rho}u_2'' \overline{u_1'} \\
&\quad -\Re \int \frac{1}{\rho^2}u_1' \overline{u_2'}
+\Re \int \frac{1}{\rho}u_1'' \overline{u_2'}
-\Re \int \rho u_2'' \overline{u_2'}
-\int |u_2'|^2 \\
&=-\frac{1}{2}\left . \frac{|u_1'(\rho)|^2}{\rho} \right |_0^1-\frac{1}{2}\int \frac{1}{\rho^2}|u_1'|^2
+\Re \left.\frac{u_2'(\rho)\overline{u_1'(\rho)}}{\rho} \right |_0^1-\Re \int \frac{1}{\rho}u_2' \overline{u_1''}+\Re \int \frac{1}{\rho^2}u_2' \overline{u_1'} \\
&\quad -\Re \int \frac{1}{\rho^2}u_1' \overline{u_2'}
+\Re \int \frac{1}{\rho}u_1'' \overline{u_2'}
-\frac{1}{2} |u_2'(1)|^2-\frac{1}{2}\int |u_2'|^2 \\
&\leq -\frac{1}{2}\|\mb{u}\|^2
\end{align*}
by Cauchy's inequality. Note carefully that the apparently singular boundary terms at $\rho=0$ vanish by de l'H\^ opital's rule thanks to the boundary condition $u_1''(0)=0$.

Now let $\mb{f} \in C^\infty_c(0,1] \times C^\infty_c(0,1]$ and define
$$ f(\rho):=\frac{f_1(\rho)}{\rho^2}+f_2(\rho). $$
Then $f \in C^\infty_c(0,1]$.
Set
$$ u_2(\rho):=\frac{\rho}{1-\rho^2}\int_\rho^1 f(\rho')d\rho'. $$
Obviously, $u_2 \in C^\infty[0,1)$ and $u_2(0)=0$.
Expanding $f$ around $\rho=1$ yields
$$ f(\rho)=f(1)-f'(1)(1-\rho)+\frac{f''(1)}{2}(1-\rho)^2+R(\rho) $$
where the remainder $R \in C^\infty[0,1]$ satisfies the estimate $|R(\rho)|\leq C(1-\rho)^3$ for all $\rho \in [0,1]$.
Performing an analogous expansion for $f'$ and $f''$, one sees that $|R^{(k)}(\rho)|\leq C_k (1-\rho)^{3-k}$ for $k=0,1,2,3$ and all $\rho \in [0,1]$.
Thus, for $\tilde{R}(\rho):=\int_\rho^1 R(\rho')d\rho'$, we have the estimate
$$ \left |\frac{d^k}{d\rho^k}\frac{\tilde{R}(\rho)}{1-\rho}\right |\leq C_k (1-\rho)^{3-k} $$
for $k=0,1,2,3$ and $\rho \in [0,1]$ by the product rule.
This shows that $u_2$ can be written as
$$ u_2(\rho)=\frac{\rho}{1+\rho}\left [f(1)-\frac{f'(1)}{2}(1-\rho)+\frac{f''(1)}{6}(1-\rho)^2+\frac{\tilde{R}(\rho)}{1-\rho} \right ]$$
which implies $u_2 \in C^3[0,1]$.
Furthermore, we set
$$ u_1(\rho):=\rho^2 u_2(\rho)-\int_0^\rho \rho' \left [u_2(\rho')+f_2(\rho') \right ]d\rho' $$
and the above implies $u_1 \in C^3[0,1]$. Clearly, we have $u_1(0)=u_1'(0)=u_1''(0)=0$ and putting everything together we see that $\mb{u}=(u_1, u_2)\in \mc{D}(\tilde{L}_0)$.
By straightforward differentiation one verifies that
$$ (1-\tilde{L}_0)\mb{u}=\mb{f} $$
and, since $\mb{f}$ was arbitrary, we conclude that the range of $(1-\tilde{L}_0)$ is dense in $\mc{H}$ by Lemma \ref{lem:dense}.
The claim now follows from the Lumer--Phillips Theorem, see \cite{engel}, p.~83, Theorem 3.15.
 \end{proof}

\begin{remark}
 Note that, strictly speaking, we have not solved the free equation \eqref{eq:freeoppre} but a slightly more general problem where $\tilde{L}_0$ is substituted by its closure $L_0$. This is typical for the approach to PDEs with operator methods.
\end{remark}

\begin{remark}
The operator $S_0(\tau)$ is bounded for any $\tau \geq 0$. Therefore, we can apply it to general elements of $\mc{H}$ and not only to elements in the subspace $\mc{D}(L_0)$.
By doing so, we obtain more general ``solutions'' of the free equation.
Although those ``solutions'' do not actually solve the free equation (they solve it only in a weak sense), they are called \emph{mild solutions}, cf.~\cite{engel}, p.~146, Definition 6.3.
\end{remark}

We immediately obtain a result on the structure of the spectrum $\sigma(L_0)$ which will be useful later on.

\begin{corollary}
\label{cor:specL0}
For the spectrum $\sigma(L_0)$ of the generator $L_0$ we have
$$ \sigma(L_0)\subset \left \{\lambda \in \mathbb{C}: \Re \lambda \leq -\tfrac{1}{2} \right \} $$
and the resolvent of $L_0$ satisfies the estimate
$$ \|R_{L_0}(\lambda)\|\leq \frac{1}{\Re \lambda+\frac{1}{2}} $$
for all $\lambda \in \mathbb{C}$ with $\Re \lambda>-\frac{1}{2}$.
\end{corollary}

\begin{proof}
 This follows from the growth estimate $\|S_0(\tau)\|\leq e^{-\frac{1}{2}\tau}$ in Proposition \ref{prop:gen0} and \cite{engel}, p.~55, Theorem 1.10.
\end{proof}

By a standard result from semigroup theory we can now conclude the well--posedness of the linearized problem.

\begin{corollary}
\label{cor:gen}
 The operator $L_0+L'$ generates a strongly continuous one--parameter semigroup $S: [0,\infty) \to \mc{B}(\mc{H})$ satisfying
 $$ \|S(\tau)\|\leq e^{(-\frac{1}{2}+\|L'\|)\tau} $$
 for all $\tau \geq 0$.
 In particular, the Cauchy problem
\begin{equation*}
 \left \{ \begin{array}{l}
\frac{d}{d\tau}\Phi(\tau)=(L_0+L')\Phi(\tau) \mbox{ for }\tau>-\log T \\
\Phi(-\log T)=\mb{u}_T
          \end{array} \right .
\end{equation*}
for $\mb{u}_T \in \mc{D}(L_0)$ has a unique solution which is given by
$$ \Phi(\tau)=S(\tau+\log T)\mb{u}_T $$
for all $\tau \geq -\log T$.
\end{corollary}

\begin{proof}
 This follows from the bounded perturbation theorem, see \cite{engel}, p.~158, Theorem 1.3.
\end{proof}

\subsection{Spectral analysis of the generator}

We analyse the spectrum of the generator $L_0+L'$ in more detail which will eventually allow us to sharpen the growth estimate given in Corollary \ref{cor:gen}.
A crucial observation in this respect is the fact that the perturbation $L'$ is compact.

\begin{lemma}
\label{lem:L'compact}
 The operator $L': \mc{H} \to \mc{H}$ is compact.
\end{lemma}

\begin{proof}
We define auxiliary operators $K$ on $C^\infty_c(0,1]$
and $A$, $M$ on $C^\infty_c(0,1] \times C^\infty_c(0,1]$ by
$$ (Ku)(\rho):=\int_0^\rho \rho' u(\rho')d\rho', \quad A\mb{u}:=\left (\begin{array}{c}Ku_2 \\ 0 \end{array} \right ), \quad M\mb{u}:=\left (\begin{array}{c}-V_1 u_1 \\ 0 \end{array} \right ). $$
Then we have (cf.~Lemma \ref{lem:HinC})
\begin{equation}
\label{eq:proofL'compact}
 \|A\mb{u}\|^2=\|Ku_2\|_1^2=\int_0^1 \frac{|(Ku_2)'(\rho)|^2}{\rho^2}d\rho=\|u_2\|_{L^2(0,1)}^2 \leq \|u_2'\|_{L^2(0,1)}^2=\|u_2\|_2^2 \leq \|\mb{u}\|^2 
 \end{equation}
 and, by using Hardy's inequality,
\begin{align*} 
\|M\mb{u}\|^2&=\|V_1 u_1\|_1^2=\int_0^1 \frac{|V_1'(\rho)u_1(\rho)+V_1(\rho)u_1'(\rho)|^2}{\rho^2}d\rho \\
&\lesssim \|V_1'\|_{L^\infty(0,1)}^2 \int_0^1 \frac{|u_1(\rho)|^2}{\rho^2}d\rho+\|V_1\|_{L^\infty(0,1)}^2 \int_0^1 \frac{|u_1'(\rho)|^2}{\rho^2}d\rho \\
&\lesssim \int_0^1 |u_1'(\rho)|^2 d\rho + \|u_1\|_1^2 \\
&\lesssim \|\mb{u}\|^2
\end{align*}
for all $\mb{u} \in C^\infty_c(0,1] \times C^\infty_c(0,1]$.
Consequently, by Lemma \ref{lem:dense}, $K$, $A$ and $M$ extend to bounded operators 
$K: H_2 \to H_1$, $A: \mc{H} \to \mc{H}$ and $M: \mc{H} \to \mc{H}$.
By construction, we have $L'=MA$ and therefore, by an elementary result of functional analysis, it suffices to prove compactness of $A$.

Now recall from the proof of Lemma \ref{lem:HinC} that
$\|u\|_{H^1(0,1)}\lesssim \|u\|_2$ for all $u \in H_2$ and suppose that $(\mb{u}_j) \subset \mc{H}$ is a bounded sequence. 
This implies that $(u_{2j})$ is a bounded sequence in $H_2$ and it follows that $(u_{2j})$ is a bounded sequence in $H^1(0,1)$. By the compact embedding $H^1(0,1) \subset \subset L^2(0,1)$ we conclude that $(u_{2j})$ has a subsequence which converges in $L^2(0,1)$.
Consequently, Eq.~\eqref{eq:proofL'compact} implies that $(A \mb{u}_j)$ has a convergent subsequence in $\mc{H}$ which proves the compactness of $A$.
\end{proof}

Lemma \ref{lem:L'compact} implies that adding the perturbation $L'$ only affects the point spectrum as will be shown next.
From now on we write $L:=L_0+L'$.

\begin{lemma}
\label{lem:specL}
If $\lambda \in \sigma(L)\backslash \sigma(L_0)$ then $\lambda \in \sigma_p(L)$.
\end{lemma}

\begin{proof}
The proof is actually a standard argument but in order to keep things as self--contained as possible, we present it in full detail.
Let $\lambda \in \sigma(L)\backslash \sigma(L_0)$. 
The well--known identity 
$$\lambda-L=(1-L'R_{L_0}(\lambda))(\lambda-L_0)$$
 implies that $1 \in \sigma(L'R_{L_0}(\lambda))$ since otherwise the operator $\lambda-L$ would be bounded invertible which contradicts $\lambda \in \sigma(L)$.
However, as a composition of a compact (see Lemma \ref{lem:L'compact}) and a bounded operator, $L'R_{L_0}(\lambda)$ is compact and it follows that $1 \in \sigma_p(L'R_{L_0}(\lambda))$ by the spectral theorem for compact operators (see e.g.,~\cite{kato}, p.~185, Theorem 6.26).
As a consequence, there exists a nonzero $\mb{f} \in \mc{H}$ such that $L'R_{L_0}(\lambda)\mb{f}=\mb{f}$.
Set $\mb{u}:=R_{L_0}(\lambda)\mb{f}$.
Then $\mb{u} \in \mc{D}(L_0)=\mc{D}(L)$, $\mb{u} \not=0$, and 
$(\lambda-L)\mb{u}=0$ which shows that $\lambda \in \sigma_p(L)$.
\end{proof}

Before we can study the spectrum of $L$ in detail, we need one more technical lemma. 
The problem is that, strictly speaking, we do not know how the operator $L$ acts on general $\mb{u} \in \mc{D}(L)$.
We only know that its action on $\mb{u} \in \mc{D}(\tilde{L}_0)$ is given by $(\tilde{L}_0+L')\mb{u}$.
However, the operator $L_0$ (and consequently $L$) is given by the closure of $\tilde{L}_0$ which is a priori an abstract object.
The following result closes this gap.

\begin{lemma}
\label{lem:actionL}
Let $\mb{u} \in \mc{D}(L)$. Then $\mb{u} \in C^1(0,1) \times C^1(0,1) \cap C[0,1] \times C[0,1]$, $\mb{u}(0)=0$ and $(\lambda-L)\mb{u}=\mb{f}$ for $\mb{f} \in \mc{H}$ implies
$$ 
u_1(\rho)=\rho^2 u_2(\rho)+(\lambda-2)\int_0^\rho \rho' u_2(\rho')d\rho'-\int_0^\rho \rho' f_2(\rho')d\rho' $$
and
\begin{align}
\label{eq:actionL}
-(1-\rho^2)&u''(\rho)-2\frac{1-\rho^2}{\rho}u'(\rho)+2 \lambda \rho u'(\rho)+\lambda (\lambda+1)u(\rho)+\frac{V(\rho)}{\rho^2}u(\rho) \\ &=\frac{f_1(\rho)}{\rho^2}+f_2(\rho)+\frac{\lambda-1}{\rho^2}\int_0^\rho \rho' f_2(\rho')d\rho' \nonumber
\end{align}
for all $\rho \in (0,1)$ where $u \in C^2(0,1)$ is defined by
$$ u(\rho):=\frac{1}{\rho^2}\int_0^\rho \rho' u_2(\rho')d\rho'. $$
\end{lemma}

\begin{proof}
 Let $\mb{u} \in \mc{D}(L)=\mc{D}(L_0)$. Since $\mb{u} \in \mc{H}$, Lemma \ref{lem:HinC} implies $\mb{u} \in C[0,1] \times C[0,1]$ and $\mb{u}(0)=0$.
 By definition of the closure, there exists a sequence $(\mb{u}_j) \subset \mc{D}(\tilde{L}_0) \subset C^1[0,1] \times C^1[0,1]$ such that $\mb{u}_j \to \mb{u}$ and $\tilde{L}_0 \mb{u}_j \to L_0 \mb{u}$ in $\mc{H}$.
 This implies $(1-\tilde{L}_0)\mb{u}_j \to (1-L_0)\mb{u}$ in $\mc{H}$ and we set $\mb{f}_j:=(1-\tilde{L}_0)\mb{u}_j$ and $\mb{f}:=(1-L_0)\mb{u}$.
 Then $\mb{f} \in \mc{H}$ and $\mb{f}_j \to \mb{f}$.
 Explicitly, $(1-\tilde{L}_0)\mb{u}_j=\mb{f}_j$ reads
 \begin{equation}
 \label{eq:proofactionL}
\left \{ \begin{array}{l}
\rho u_{1j}'(\rho)-\rho u_{2j}'(\rho)+u_{2j}(\rho)=f_{1j}(\rho)\\
-\frac{1}{\rho}u_{1j}'(\rho)+\rho u_{2j}'(\rho)+u_{2j}(\rho)=f_{2j}(\rho)
            \end{array} \right . 
\end{equation}
for all $\rho \in [0,1]$.
Multiplying the second equation in \eqref{eq:proofactionL} by $\rho^2$ and adding the result to the first equation we obtain
$$ -\rho (1-\rho^2)u_{2j}'(\rho)+(1+\rho^2)u_{2j}(\rho)=f_{1j}(\rho)+\rho^2 f_{2j}(\rho). $$
According to Lemma \ref{lem:HinC}, convergence in $\mc{H}$ implies uniform convergence of the individual components. Therefore, $u_{2j} \to u_2$, $f_{1j} \to f_1$ and $f_{2j} \to f_2$ in $L^\infty(0,1)$ and thus, the above equation implies that $u_{2j}'$ converges in $L^\infty(a,b)$ for any $[a,b] \subset (0,1)$.
Note that we cannot include the endpoints due to the singular factors in front of $u_{2j}'$.
Consequently, both $u_{2j}$ and $u_{2j}'$ converge in $L^\infty(a,b)$ and by an elementary result of analysis, this implies that the limiting function $u_2$ belongs to $C^1[a,b]$ and $u_{2j}' \to u_2'$.
Since $[a,b] \subset (0,1)$ was arbitrary, we obtain $u_2 \in C^1(0,1)$ and
the first equation in \eqref{eq:proofactionL} immediately implies $u_1 \in C^1(0,1)$ as claimed.

As a consequence, the operator $L$ acts as a classical differential operator on $\mb{u} \in \mc{D}(L)$ and $(\lambda-L)\mb{u}=\mb{f}$ for an $\mb{f} \in \mc{H}$ implies
\begin{equation}
\label{eq:proofactionL2}
\left \{ \begin{array}{l}
\rho u_1'(\rho)+(\lambda-1)u_1(\rho)-\rho u_2'(\rho)+u_2(\rho)+V_1(\rho)\int_0^\rho \rho' u_2(\rho')d\rho' 
=f_1(\rho)\\
-\frac{1}{\rho}u_1'(\rho)+\rho u_2'(\rho)+\lambda u_2(\rho)=f_2(\rho)
            \end{array} \right . 
\end{equation}
for all $\rho \in (0,1)$.
Thanks to the boundary condition $u_1(0)=0$, the second equation in \eqref{eq:proofactionL2} necessarily implies
$$
u_1(\rho)=\rho^2 u_2(\rho)+(\lambda-2)\int_0^\rho \rho' u_2'(\rho')d\rho'-\int_0^\rho \rho' f_2(\rho')d\rho' $$
and inserting this into the first equation in \eqref{eq:proofactionL2} we obtain
\begin{align*} 
-\rho (1-\rho^2)u_2'(\rho)&+u_2(\rho)+(2\lambda-1)\rho^2 u_2(\rho)+(\lambda-1)(\lambda-2)\int_0^\rho \rho' u_2(\rho')d\rho' \\
&+V_1(\rho)\int_0^\rho \rho' u_2(\rho')d\rho'=f_1(\rho)+\rho^2 f_2(\rho)+(\lambda-1)\int_0^\rho \rho' f_2(\rho')d\rho'. \nonumber
\end{align*}
Defining $u \in C^2(0,1)$ by
$$ u(\rho):=\frac{1}{\rho^2}\int_0^\rho \rho' u_2(\rho')d\rho' $$ we see that
$u$ satisfies Eq.~\eqref{eq:actionL} where we use
$V(\rho)=\rho^2 V_1(\rho)+2$, cf.~Eq.~\eqref{eq:V1}.
\end{proof}

Now we are ready to establish the connection between the spectrum of the generator $L$ and the mode stability analysis from Sec.~\ref{sec:mode}.

\begin{proposition}
\label{prop:specL}
 Suppose $\lambda \in \sigma_p(L)$ and $\mathrm{Re}\lambda \geq -\frac{1}{2}$. Then there exists a nonzero function $u(\cdot,\lambda) \in C^\infty[0,1]$ that solves Eq.~\eqref{eq:evode}.
\end{proposition}

\begin{proof}
 Let $\lambda \in \sigma_p(L)$ with $\Re \lambda\geq -\frac{1}{2}$ and suppose $\mb{u} \in \mc{D}(L_0)$ is an associated eigenfunction.
 According to Lemma \ref{lem:actionL}, $(\lambda-L)\mb{u}=0$ implies
\begin{equation} 
\label{eq:proofL'sigmap3}
u_1(\rho)=\rho^2 u_2(\rho)+(\lambda-2)\int_0^\rho \rho' u_2(\rho')d\rho'  
\end{equation}
and 
$$ -(1-\rho^2)u''(\rho)-2\frac{1-\rho^2}{\rho}u'(\rho)+2 \lambda \rho u'(\rho)+\lambda (\lambda+1)u(\rho)+\frac{V(\rho)}{\rho^2}u(\rho)=0 $$
for 
$$ u(\rho):=\frac{1}{\rho^2}\int_0^\rho \rho' u_2(\rho')d\rho'. $$
Consequently, $u \in C^2(0,1)$ satisfies Eq.~\eqref{eq:evode} for $\rho \in (0,1)$.
Note that Eq.~\eqref{eq:proofL'sigmap3} also implies that $u_2$ (and therefore $u$) is nonzero because otherwise, we would have $\mb{u}=0$, a contradiction to the assumption that $\mb{u}$ is an eigenfunction.
It remains to show that in fact $u \in C^\infty[0,1]$. 
By standard ODE theory it follows immediately that $u \in C^\infty(0,1)$ since the coefficients of the equation are smooth on $(0,1)$.
Furthermore, by Frobenius' method (see e.g., \cite{miller}, \S 6.2) we know the asymptotic behavior of solutions to Eq.~\eqref{eq:evode}.
Around $\rho=0$, the analytic solution behaves like $\rho$ and the nonanalytic one like $\rho^{-2}$, see Sec.~\ref{sec:mode}.
However, by de l'H\^ opital's rule we obtain
$$ \lim_{\rho \to 0+}u(\rho)=\lim_{\rho \to 0+}\frac{\int_0^\rho \rho' u_2(\rho')d\rho'}{\rho^2}
=\lim_{\rho \to 0+}\frac{u_2(\rho)}{2}=0 $$
due to the boundary condition $u_2(0)=0$, see Lemma \ref{lem:HinC}, and 
this already implies that $u$ is analytic around $\rho=0$ which shows $u \in C^\infty[0,1)$.
Furthermore, around $\rho=1$, the analytic solution behaves like $1$ and the nonanalytic one like $(1-\rho)^{1-\lambda}$, at least if $\lambda \notin \{1, 0, -1, -2, \dots\}$ which we shall assume for the moment.
Now suppose $u$ is not smooth at $\rho=1$. This implies that 
$u$ can be written as 
$$ u(\rho)=(1-\rho)^{1-\lambda}\sum_{j=0}^\infty a_j (1-\rho)^j+h(\rho) $$
for suitable constants $a_j \in \mathbb{C}$, $j \in \mathbb{N}_0$, and a suitable function $h$ which is analytic around $\rho=1$. Furthermore, $a_0 \not=0$ and the series has a positive convergence radius (see \cite{miller}, \S 6.2).
This shows in particular that $u''(\rho)\simeq (1-\rho)^{-1-\lambda}$ as $\rho \to 1-$.
Note that $u$ has to satisfy the integrability condition
\begin{equation} 
\label{eq:proofL'sigmap4}
\int_{1-\varepsilon}^1 \left |\left (\frac{1}{p}(p^2 u)' \right )'(\rho) \right |^2 d\rho=\int_{1-\varepsilon}^1 |u_2'(\rho)|^2 d\rho \leq \|u_2\|_2^2 < \infty 
\end{equation}
for any $\varepsilon \in (0,1)$ where $p(\rho):=\rho$.
However, we have $(\frac{1}{p}(p^2 u)')'(\rho)\simeq (1-\rho)^{-1-\lambda}$ and this implies
$$ \int_{1-\varepsilon}^1 \left |\left (\frac{1}{p}(p^2 u)' \right )'(\rho) \right |^2 d\rho \simeq \int_{1-\varepsilon}^1 |1-\rho|^{-2-2\Re \lambda} d\rho=\infty $$
since we assume $\Re \lambda \geq -\frac{1}{2}$.
This is a contradiction to \eqref{eq:proofL'sigmap4} and therefore, we must have $u \in C^\infty[0,1]$ provided that $\lambda \notin \{0, 1\}$.
If $\lambda \in \{0,1\}$ and $u$ is not smooth at $\rho=1$, Frobenius' method implies that $u$ is of the form
$$ u(\rho)=\sum_{j=0}^\infty a_j (1-\rho)^j+c(1-\rho)^{1-\lambda}\log(1-\rho)\sum_{j=0}^\infty b_j (1-\rho)^j $$
for suitable constants $a_j, b_j, c \in \mathbb{C}$, $j \in \mathbb{N}_0$, where $a_0 \not=0$, $b_0 \not=0$ and the series have positive convergence radii (cf.~\cite{miller}, \S 6.2).
Since we assume that $u$ is not smooth at $\rho=1$, we must have $c\not= 0$ and this implies $u''(\rho)\simeq (1-\rho)^{-1-\lambda}$ as $\rho \to 1-$ which, as before, contradicts \eqref{eq:proofL'sigmap4}.
Therefore, in any case, we conclude that $u\in C^\infty[0,1]$ which finishes the proof.
\end{proof}

Now we are ready to obtain a sufficiently accurate description of the spectrum of $L$.
To this end, we define the \emph{spectral bound} for the mode stability problem 
Eq.~\eqref{eq:evodedual} as follows.
\begin{definition}
\label{def:sb}
The \emph{spectral bound} $s_0$ is defined as 
$$ s_0:=\sup\{\mathrm{Re}\lambda: \mbox{Eq.~\eqref{eq:evodedual} has a nontrivial solution $v(\cdot,\lambda) \in C^\infty[0,1]$}\}. $$
We say that the fundamental self--similar solution $\psi^T$ is \emph{mode stable} iff $s_0<0$.
\end{definition}
Note that Proposition \ref{prop:modstab} implies that $s_0<\frac{1}{2}$.

\begin{lemma}
 \label{lem:specL2}
 For the spectrum $\sigma(L)$ of $L$ we have
 $$ \sigma(L) \subset \left \{\lambda \in \mathbb{C}: \mathrm{Re}\lambda \leq \max \left \{-\tfrac{1}{2},s_0 \right \} \right \} \cup \{1\} $$
 where $s_0$ is the spectral bound, see Definition \ref{def:sb}.
 Furthermore, $1 \in \sigma_p(L)$ and the associated (geometric) eigenspace is one--dimensional and spanned by 
 $$ \mb{g}(\rho):=\frac{1}{(1+\rho^2)^2}\left ( \begin{array}{c}
2 \rho^3  \\ \rho (3+\rho^2) 
                          \end{array} \right ). $$
\end{lemma}

\begin{proof}
 Set $M:=\{\lambda \in \mathbb{C}: \mathrm{Re}\lambda \leq \max \{-\frac{1}{2},s_0 \} \} \cup \{1\}$ and let $\lambda \in \sigma(L)$.
 If $\Re \lambda \leq -\frac{1}{2}$ then $\Re \lambda \leq \max\{-\frac{1}{2},s_0\}$ and we have $\lambda \in M$.
 If $\Re \lambda > -\frac{1}{2}$ then Corollary \ref{cor:specL0} implies that $\lambda \notin \sigma(L_0)$. Thus, according to Lemma \ref{lem:specL}, $\lambda \in \sigma_p(L)$.
 Consequently, Proposition \ref{prop:specL} shows that there exists a nontrivial $u(\cdot,\lambda) \in C^\infty[0,1]$ that solves Eq.~\eqref{eq:evode}.
 If $\lambda=1$ then trivially $\lambda \in M$.
 Thus, we assume $\lambda \not= 1$.
 Then Lemma \ref{lem:dual} implies that Eq.~\eqref{eq:evodedual} has a nontrivial smooth solution and by definition of $s_0$ we conclude that $s_0 \geq \Re \lambda > -\frac{1}{2}$.
This shows that $\Re\lambda \leq s_0=\max\{-\frac{1}{2},s_0\}$ and therefore, $\lambda \in M$.

We have $\mb{g} \in \mc{D}(\tilde{L}_0)$ which implies $L \mb{g}=(\tilde{L}_0+L')\mb{g}$ and by straightforward differentiation one obtains $(1-\tilde{L}_0-L')\mb{g}=0$. This shows that $1 \in \sigma_p(L)$ and $\mb{g}$ is an associated eigenfunction.
Now suppose $\tilde{\mb{g}}$ is another eigenfunction of $L$ with eigenvalue $1$.
Invoking Lemma \ref{lem:actionL}, we see that
\begin{equation}
\label{eq:prooflemspecL22}
\tilde{g}_1(\rho)=\rho^2 \tilde{g}_2(\rho)-\int_0^\rho \rho' \tilde{g}_2(\rho')d\rho' 
\end{equation}
and 
\begin{equation}
 \label{eq:prooflemspecL2}
-(1-\rho^2)\tilde{g}''(\rho)-2\frac{1-\rho^2}{\rho}\tilde{g}'(\rho)+2 \lambda \rho \tilde{g}'(\rho)+\lambda (\lambda+1)\tilde{g}(\rho)+\frac{V(\rho)}{\rho^2}\tilde{g}(\rho)=0 
\end{equation}
where $\tilde{g}(\rho):=\frac{1}{\rho^2}\int_0^\rho \rho' \tilde{g}_2(\rho')d\rho'$.
Furthermore, $\tilde{g} \in C^\infty[0,1]$.
However, by Frobenius' method we know that the only nontrivial solution to Eq.~\eqref{eq:prooflemspecL2} in $C^\infty[0,1]$ is a multiple of
$$ g(\rho):=\frac{1}{\rho^2}\int_0^\rho \rho' g_2(\rho')d\rho'=\frac{\rho}{1+\rho^2} $$ and consequently, there exists a $c \in \mathbb{C}$ such that
$$ \frac{1}{\rho^2}\int_0^\rho \rho' \tilde{g}_2(\rho')=\tilde{g}(\rho)=c g(\rho)=\frac{c}{\rho^2}\int_0^\rho \rho' g_2(\rho') $$
which implies $\tilde{g}_2(\rho)=c g_2(\rho)$.
Inserting this in Eq.~\eqref{eq:prooflemspecL22} we obtain $\tilde{g}_1(\rho)=c g_1(\rho)$ and therefore, $\tilde{\mb{g}}=c\mb{g}$. This proves that the geometric multiplicity of the eigenvalue $1 \in \sigma_p(L)$ is $1$.
\end{proof}

\subsection{Projection on the stable subspace}
\label{sec:gauge}
As already discussed in Sec.~\ref{sec:mode}, the unstable eigenvalue $1 \in \sigma_p(L)$ is a direct consequence of the time translation symmetry of the original problem. 
In this sense, the instability is artificial and we are actually interested in the time evolution with this gauge instability removed.
By Lemma \ref{lem:specL2}, the eigenspace associated to the eigenvalue $1 \in \sigma_p(L)$ is spanned by a single vector (the \emph{gauge mode}) which henceforth will be denoted by $\mb{g} \in \mc{D}(L)$.
We intend to construct a projection that removes the gauge instability.
However, we cannot simply use an orthogonal projection because this is not compatible with the time evolution (the operator $L$ is not self--adjoint).
Instead, we need to construct a projection that commutes with the semigroup $S(\tau)$.
This can be done in a completely abstract way.
To this end, let us recall some well known facts from spectral theory, see e.g., \cite{kato}.
At this point we have to warn the reader that our definition of the resolvent $R_L(\lambda)$ differs from \cite{kato} by a sign, i.e., we have $R_L(\lambda)=(\lambda-L)^{-1}$ whereas in \cite{kato} the 
convention $R_L(\lambda)=(L-\lambda)^{-1}$ is used.
This has to be taken into account in the following.

First of all, the resolvent $R_L$ is a holomorphic function \footnote{Of course, this function has values in the Banach space $\mc{B}(\mc{H})$ of bounded linear operators on $\mc{H}$. However, it turns out that this is no obstacle to developing complex analysis just along the same lines as for functions taking values in $\mathbb{C}$, see \cite{kato}.} on $\rho(L)$.
This follows from the expansion of $R_L(\lambda)$ in a Neumann series, see \cite{kato}, p.~174, Theorem 6.7.
According to Lemma \ref{lem:specL2}, the gauge eigenvalue $1$ is an isolated point in the spectrum of $L$ which implies that $R_L$ has an isolated singularity at $1$.
Therefore, we can find a curve $\Gamma$ which lies completely in $\rho(L)$ and encloses the point $1$. For instance, we may simply take a circle with radius $\frac{1}{3}$ and center $1$ in the complex plane, see Fig.~\ref{fig:spec}.
\begin{figure}[h]
 \includegraphics[totalheight=5cm]{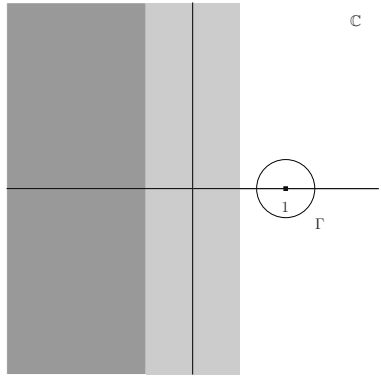}
 \caption{The structure of the spectrum $\sigma(L)$. According to numerics, the lighter shaded region belongs to the resolvent set $\rho(L)$. However, this has not been proved rigorously so far.
On the contrary, the darker shaded region contains the spectrum of the free generator $L_0$.}
  \label{fig:spec}
\end{figure}
Consequently, by a well--known result from complex analysis, $R_L$ can be expanded in a Laurent series
$$ R_L(\lambda)=\sum_{k=-\infty}^\infty (\lambda-1)^k A_k $$
around $1$ where the ``coefficients'' $A_k$ (which are actually operators here) are given by the formula
$$ A_k=\frac{1}{2\pi i}\int_\Gamma (\lambda-1)^{-k-1}R_L(\lambda)d\lambda. $$
This expression makes perfect sense since the integral is taken over a continuous function (with values in a Banach space, though) and
it may be defined simply via Riemann sums, just like in the scalar case.
Consequently, each $A_k$ is a bounded linear operator on $\mc{H}$. 
From this formula and the resolvent equation 
$R_L(\lambda)-R_L(\mu)=(\mu-\lambda)R_L(\lambda)R_L(\mu)$, valid for all $\lambda, \mu \in \rho(L)$, it follows that $A_{-1}^2=A_{-1}$, see \cite{kato}, p.~38--39.
In other words, $P:=A_{-1}$ is a projection. 
Explicitly, we have
$$ P=\frac{1}{2\pi i}\int_\Gamma R_L(\lambda) d\lambda. $$
Furthermore, from the resolvent equation above and this formula, it follows immediately that $P$ commutes with $R_L(\lambda)$ for any $\lambda \in \rho(L)$.
This implies $PL \subset LP$ via \cite{kato}, p.~173, Theorem 6.5.
We define the subspaces $\mc{M}:=P\mc{H}$ and $\mc{N}:=(1-P)\mc{H}$.
Note that both $\mc{M}$ and $\mc{N}$ are closed since $\mc{M}=\ker(1-P)$ and $\mc{N}=\ker P$ and $P$ is bounded.
This shows that $\mc{M}$ and $\mc{N}$ equipped with $(\cdot|\cdot)$ are Hilbert spaces.
However, the most important feature of the projection $P$ is the fact that it decomposes the operator $L$ in the following way.
Let $L_\mc{N}$ be defined by $\mc{D}(L_\mc{N}):=\mc{D}(L)\cap \mc{N}$ and $L_\mc{N}\mb{u}:=L \mb{u}$.
Since $P$ is idempotent, we have for $\mb{u} \in \mc{D}(L_\mc{N})$ that $\mb{u}=(1-P)\mb{u}$ and thus, $L\mb{u}=L(1-P)\mb{u}=(1-P)L\mb{u} \in \mc{N}$.
Hence, $L_\mc{N}$ may be regarded as a linear operator on the Hilbert space $\mc{N}$.
Of course, an analogous statement is true for $L_\mc{M}$ on $\mc{M}$ and from \cite{kato}, p.~178, Theorem 6.17, we have $\sigma(L_\mc{M})=\{1\}$ and $\sigma(L_\mc{N})=\sigma(L)\backslash \{1\}$.
Loosely speaking, $L_\mc{N}$ is ``the operator $L$ with the gauge instability removed''. Consequently, we call $\mc{N}$ the \emph{stable subspace}.

\begin{lemma}
\label{lem:genSN}
 The operator $L_\mc{N}: \mc{D}(L_\mc{N})\subset \mc{N} \to \mc{N}$ generates a strongly continuous one--parameter semigroup $S_\mc{N}: [0, \infty) \to \mc{B}(\mc{N})$ satisfying
 $$ \|S_\mc{N}(\tau)\|\leq e^{(-\frac{1}{2}+\|L'\|)\tau} $$
 for all $\tau \geq 0$.
\end{lemma}

\begin{proof}
Note first that $L_\mc{N}$ is densely defined.
To see this, let $\mb{f} \in \mc{N} \subset \mc{H}$ be arbitrary.
Since $\mc{D}(L)$ is dense in $\mc{H}$, 
it follows that there exists a sequence $(\mb{u}_j)\subset \mc{D}(L)$ such that $\mb{u}_j \to \mb{f}$ in $\mc{H}$ as $j \to \infty$.
Furthermore, we have $PL \subset LP$ which in particular implies that $P\mc{D}(L) \subset \mc{D}(L)$ and thus, $(P\mb{u}_j) \subset \mc{N} \cap \mc{D}(L)=\mc{D}(L_\mc{N})$.
By the boundedness of $P$, we obtain $\|P\mb{u}_j-P\mb{f}\| \lesssim \|\mb{u}_j-\mb{f}\| \to 0$, i.e., $P\mb{u}_j \to P\mb{f}=\mb{f}$ in $\mc{H}$ and hence, $P\mb{u}_j \to \mb{f}$ in $\mc{N}$
as $j \to \infty$.
Thus, $\mc{D}(L_\mc{N})$ is dense in $\mc{N}$ as claimed.

Moreover, let $(\mb{u}_j) \subset \mc{D}(L_\mc{N})$ be such that $\mb{u}_j \to \mb{u}$ in $\mc{N}$ and $L_\mc{N}\mb{u}_j \to \mb{f}$ as $j \to \infty$ for some $\mb{f} \in \mc{N}$.
This implies $\mb{u}_j \to \mb{u}$, $L\mb{u}_j \to \mb{f}$ as $j \to \infty$ in $\mc{H}$ and by the closedness of $L$, we obtain $\mb{u} \in \mc{D}(L)$ and $L\mb{u}=\mb{f}$.
Consequently, we have $\mb{u} \in \mc{D}(L)\cap \mc{N}=\mc{D}(L_\mc{N})$ and $L_\mc{N}\mb{u}=L\mb{u}=\mb{f}$ which shows that $L_\mc{N}$ is closed.

Now recall the estimate from the proof of Proposition \ref{prop:gen0} which readily implies that
\begin{align} 
\label{eq:proofgenSN}
\Re (L\mb{u}|\mb{u})&=\Re (L_0\mb{u}|\mb{u})+\Re (L'\mb{u}|\mb{u})=\Re(\tilde{L}_0 \mb{u}|\mb{u})+\Re (L' \mb{u}|\mb{u})\leq -\frac{1}{2}\|\mb{u}\|^2+\|L'\mb{u}\|\|\mb{u}\| \\
&\leq \left (-\frac{1}{2}+\|L'\| \right )\|\mb{u}\|^2 \nonumber
 \end{align}
for all $\mb{u} \in \mc{D}(\tilde{L}_0)$.
Let $\mb{u} \in \mc{D}(L)$. By the definition of the closure, there exists a sequence $(\mb{u}_j) \subset \mc{D}(\tilde{L}_0)$ such that $\mb{u}_j \to \mb{u}$ and $L\mb{u}_j \to L\mb{u}$ as $j \to \infty$ in $\mc{H}$. 
Consequently, 
\begin{align*} |(L\mb{u}_j|\mb{u}_j)-(L\mb{u}|\mb{u})|&=|(L\mb{u}_j|\mb{u}_j-\mb{u})+(L\mb{u}_j-L\mb{u}|\mb{u})| \\
&\leq \|L\mb{u}_j\|\|\mb{u}_j-\mb{u}\|+\|L\mb{u}_j-L\mb{u}\|\|\mb{u}\| \to 0
\end{align*}
as $j \to \infty$ which shows that 
$$(L\mb{u}|\mb{u})=\lim_{j \to \infty}(L\mb{u}_j|\mb{u}_j)\leq \lim_{j \to \infty}\left (-\frac{1}{2}+\|L'\| \right )\|\mb{u}_j\|^2=\left (-\frac{1}{2}+\|L'\| \right)\|\mb{u}\|^2 $$ and thus, the estimate \eqref{eq:proofgenSN} holds in fact for all $\mb{u} \in \mc{D}(L)$.
In particular, we have 
$$ \Re (L_\mc{N} \mb{u}|\mb{u})=\Re(L\mb{u}|\mb{u})\leq \left (-\frac{1}{2}+\|L'\| \right )\|\mb{u}\|^2 $$
for all $\mb{u} \in \mc{D}(L_\mc{N})$.
Finally, for any fixed $\lambda>1$ there exists the resolvent $R_L(\lambda)$ (Lemma \ref{lem:specL2}) and thus, if $\mb{f} \in \mc{N}$, we can define $\mb{u}:=R_L(\lambda)\mb{f}$. 
By definition of the resolvent, $\mb{u} \in \mc{D}(L)$ and, since $\mc{N}$ is invariant under $R_L(\lambda)$ ($P$ commutes with $R_L(\lambda)$, see above), we have in fact $\mb{u}=\mc{D}(L)\cap \mc{N}=\mc{D}(L_\mc{N})$ and 
$$ (\lambda-L_\mc{N})\mb{u}=(\lambda-L)R_L(\lambda)\mb{f}=\mb{f}. $$
Consequently, the equation $(\lambda-L_\mc{N})\mb{u}=\mb{f}$ has a solution for any $\mb{f} \in \mc{N}$ (provided that $\lambda>1$) and thus, $\lambda-L_\mc{N}$ is surjective.
Invoking the Lumer--Philips Theorem (\cite{engel}, p.~83, Theorem 3.15) yields the existence of the semigroup $S_\mc{N}$ with the stated properties.
\end{proof}

Next, we relate the semigroup $S_\mc{N}$ to $S$.

\begin{corollary}
\label{cor:SNS}
The stable subspace $\mc{N}$ is invariant under $S(\tau)$ and $S_\mc{N}(\tau)$ is the restriction of $S(\tau)$ to $\mc{N}$, i.e., $S_\mc{N}(\tau)=S(\tau)|_\mc{N}$ for all $\tau \geq 0$.
\end{corollary}

\begin{proof}
 Fix $\tau_0>0$ and let $\mb{f} \in \mc{N}$ be arbitrary.
 According to Lemma \ref{lem:specL2}, the resolvent $R_L(\frac{n}{\tau})$ exists for all $\tau \in [0,\tau_0]$ and all $n \in \mathbb{N}$ provided that $n>\tau_0$ because in this case we have $\frac{n}{\tau}>\frac{\tau_0}{\tau}\geq 1$.
 Note further that the unique solution $\mb{u} \in \mc{D}(L_\mc{N})$ of 
$(\frac{n}{\tau}-L_\mc{N})\mb{u}=(\frac{n}{\tau}-L)\mb{u}=\mb{f}$ is given by 
$R_L(\frac{n}{\tau})\mb{f}$ (recall that $\mc{N}$ is invariant under $R_L(\frac{n}{\tau})$) and thus, $R_{L_\mc{N}}(\frac{n}{\tau})=R_L(\frac{n}{\tau})|_\mc{N}$, the restriction of the full resolvent to the subspace $\mc{N}$.
Consequently, the Post--Widder inversion formula for the Laplace transform (see \cite{engel}, p.~223, Corollary 5.5) yields
$$ S_\mc{N}(\tau)\mb{f}=\lim_{n \to \infty}\left [\frac{n}{\tau}R_{L_\mc{N}}\left (\frac{n}{\tau}\right ) \right ]^n \mb{f}=
\lim_{n \to \infty}\left [\frac{n}{\tau}R_L\left (\frac{n}{\tau}\right ) \right ]^n \mb{f}=S(\tau)\mb{f} $$
for all $\tau \in [0,\tau_0]$.
Since $\tau_0>0$ and $\mb{f} \in \mc{N}$ were arbitrary, this implies the claim.
\end{proof}

\subsection{Projection on the unstable subspace}
Heuristically speaking, the growth behavior of the semigroup $S_\mc{N}$ on the stable subspace should be dictated by the spectral bound $s_0$.
Before turning to this, however, we have to deal with the following subtlety: we do not yet know how many dimensions the \emph{unstable subspace} $\mc{M}$ has.
Ideally, we would like to have $\mc{M}=\langle \mb{g} \rangle$, i.e., the unstable subspace should be spanned by the gauge mode.
However, this does not follow from the discussion above.
In fact, since the operator $L$ is not self--adjoint, this is not true in general.
A priori, it may well be the case that $\dim \mc{M}>1$ or even $\dim \mc{M}=\infty$. 
The number $\dim \mc{M}$ is called the \emph{algebraic multiplicity} of the eigenvalue $1 \in \sigma_p(L)$ (as opposed to the \emph{geometric multiplicity}, the dimension of $\ker(1-L)$, which is $1$ by Lemma \ref{lem:specL2}).
First, we rule out the pathological case $\dim \mc{M}=\infty$.

\begin{lemma}
\label{lem:dimMfinite}
 The algebraic multiplicity of the gauge eigenvalue $1 \in \sigma_p(L)$ is finite.
\end{lemma}

\begin{proof}
 Suppose that $\dim \mc{M}=\infty$. According to \cite{kato}, p.~239, Theorem 5.28, this implies that $1$ belongs to the essential spectrum $\sigma_e(L)$ of $L$ \footnote{There are several nonequivalent definitions for the essential spectrum of a non self--adjoint operator. We use the one given in \cite{kato}.}.
 However, by \cite{kato}, p.~244, Theorem 5.35, the essential spectrum does not change under a compact perturbation and $L=L_0+L'$ where $L'$ is compact (Lemma \ref{lem:L'compact}).
 Consequently, $L$ and $L_0$ have the same essential spectrum but, according to Corollary \ref{cor:specL0}, $1$ does not belong to the spectrum of $L_0$ and, since $\sigma_e(L_0)\subset \sigma(L_0)$, it cannot belong to $\sigma_e(L_0)$ either, a contradiction. Therefore, the assumption $\dim \mc{M}=\infty$ must be false. 
\end{proof}

As a consequence of Lemma \ref{lem:dimMfinite}, $L_\mc{M}$ is in fact a finite--dimensional operator and we are left with a linear algebra problem.
Since $\sigma(L_\mc{M})=\{1\}$, the spectrum of the finite--dimensional linear operator $1-L_\mc{M}$ consists of $0$ only.
By linear algebra, this implies that $1-L_\mc{M}$ is nilpotent, i.e., there exists an $m \in \mathbb{N}$ such that $(1-L_\mc{M})^m=0$.
The following result is the key to proving $\mc{M}=\langle \mb{g}\rangle$.

\begin{lemma}
 \label{lem:gg*}
 Let $c \in \mathbb{C}\backslash \{0\}$. Then the equation $(1-L)\mb{u}=c\mb{g}$ does not have a solution $\mb{u} \in \mc{D}(L)$. 
\end{lemma}

\begin{proof}
We argue by contradiction. Without loss of generality we may set $c=1$. 
Now suppose there exists an $\mb{u} \in \mc{D}(L)$ such that
$$ (1-L)\mb{u}=\mb{g}. $$
According to Lemma \ref{lem:actionL}, this implies
\begin{equation} 
\label{eq:proofgg*} 
-(1-\rho^2)u''(\rho)-2\frac{1-\rho^2}{\rho}u'(\rho)+2 \rho u'(\rho)+2 u(\rho)+\frac{V(\rho)}{\rho^2}u(\rho)=\frac{g_1(\rho)}{\rho^2}+g_2(\rho)=:g(\rho)
\end{equation}
for all $\rho \in (0,1)$ where 
$$ u(\rho):=\frac{1}{\rho^2}\int_0^\rho \rho' u_2(\rho')d\rho'. $$
As before, de l'H\^opital's rule implies that $u(0)=0$, see the proof of Proposition \ref{prop:specL}, and we also have $u \in C[0,1]$ since $u_2 \in C[0,1]$ (Lemma \ref{lem:HinC}).
Explicitly, the right--hand side of Eq.~\eqref{eq:proofgg*} is given by
$$ g(\rho)=\frac{g_1(\rho)}{\rho^2}+g_2(\rho)=\frac{\rho(5+\rho^2)}{(1+\rho^2)^2}, $$
see Lemma \ref{lem:specL2}.
In particular, we have $g(\rho)>0$ for all $\rho \in (0,1)$.
A direct computation shows that $h_0(\rho):=\frac{\rho}{1+\rho^2}$ and 
$$ h_1(\rho)=\frac{1+9\rho^2+6 \rho^3 \log \frac{1-\rho}{1+\rho}}{3 \rho^2(1+\rho^2)} $$
are solutions of the homogeneous version of Eq.~\eqref{eq:proofgg*}.
Their Wronskian is 
$$ W(h_0, h_1)(\rho)=-\frac{1}{\rho^2(1-\rho^2)} $$
and therefore, 
the variation of constants formula shows that $u$ must be of the form
$$ u(\rho)=c_0 h_0(\rho)+c_1 h_1(\rho)-h_0(\rho)\int_{\rho_0}^\rho \rho'^2 h_1(\rho')g(\rho')d\rho'
+h_1(\rho)\int_{\rho_1}^\rho \rho'^2 h_0(\rho')g(\rho')d\rho' $$
where $c_0, c_1 \in \mathbb{C}$ and $\rho_0, \rho_1 \in [0,1]$ are constants.
The boundary condition $u(0)=0$ implies
$$ c_1=-\int_{\rho_1}^0 \rho'^2 h_0(\rho')g(\rho')d\rho', $$
i.e., the singular terms have to cancel in the limit $\rho \to 0+$ and we are left with
$$  u(\rho)=c_0 h_0(\rho)-h_0(\rho)\int_{\rho_0}^\rho \rho'^2 h_1(\rho')g(\rho')d\rho'
+h_1(\rho)\int_0^\rho \rho'^2 h_0(\rho')g(\rho')d\rho'. $$
Similarly, for $\rho \to 1-$ we have $h_1(\rho) \simeq \log (1-\rho)$ and thus, $h_1 \in L^1(\frac{1}{2},1)$ and we must have 
$$ \int_0^1 \rho'^2 h_0(\rho')g(\rho')d\rho'=0 $$
since $u \in C[0,1]$.
This, however, is impossible since the integrand satisfies $\rho^2 h_0(\rho)g(\rho)>0$ for all $\rho \in (0,1)$.
\end{proof}

Now we can prove the desired result, namely that the unstable subspace is indeed spanned by the gauge mode.

\begin{lemma}
\label{lem:M}
For any $m \in \mathbb{N}$ we have $\ker(1-L_\mc{M})^m=\ker(1-L_\mc{M})$.
As a consequence, the unstable subspace $\mc{M}$ is spanned by the gauge mode, i.e., $\mc{M}=\langle \mb{g} \rangle$, and the algebraic multiplicity of the gauge eigenvalue $1 \in \sigma_p(L)$ is $1$.
\end{lemma}

\begin{proof}
 Note first that $\mc{M} \subset \mc{D}(L)$.
 To see this, observe that
 from \cite{kato}, p.~178, Theorem 6.17 it follows that $L_\mc{M} \in \mc{B}(\mc{M})$, i.e., $L_\mc{M}$ is bounded. In our case, this can also be concluded from the fact that $\dim \mc{M}<\infty$.
 Now let $\mb{u} \in \mc{M}$.
 By the density of $\mc{D}(L) \cap \mc{M}$ in $\mc{M}$ (cf.~the proof of Lemma \ref{lem:genSN}), it follows that there exists a sequence $(\mb{u}_j)\subset \mc{D}(L) \cap \mc{M}$ such that $\mb{u}_j \to \mb{u}$ in $\mc{M}$ as $j \to \infty$.
 Since $L_\mc{M}$ is bounded, we conclude that 
 $L\mb{u}_j=L_\mc{M}\mb{u}_j \to L_\mc{M} \mb{u}$
 as $j \to \infty$ and the closedness of $L$ implies that $\mb{u} \in \mc{D}(L)$.
 This proves $\mc{M} \subset \mc{D}(L)$ as claimed.
 
 Next, we prove the assertion $\ker(1-L_\mc{M})^m=\ker(1-L_\mc{M})$ for all $m \in \mathbb{N}$.
 Obviously, this is true for $m=1$ and we proceed by induction.
 Thus, we show that $\ker(1-L_\mc{M})^m=\ker(1-L_\mc{M})$ implies 
$\ker(1-L_\mc{M})^{m+1}=\ker(1-L_\mc{M})$ for an arbitrary $m \in \mathbb{N}$.
Suppose $\mb{u} \in \ker(1-L_\mc{M})^{m+1}$, i.e., $(1-L_\mc{M})^{m+1}\mb{u}=0$.
Setting $\mb{v}:=(1-L_\mc{M})^{m-1} \mb{u}$ we conclude that $(1-L_\mc{M})^2 \mb{v}=0$ and thus, 
$(1-L_\mc{M})\mb{v}$ must be an element of $\ker(1-L_\mc{M})$.
Furthermore, since $\mc{M} \subset \mc{D}(L)$, $\mb{v} \in \mc{D}(L)$ and
by Lemma \ref{lem:specL2} we obtain
$$ (1-L)\mb{v}=(1-L_\mc{M})\mb{v}=c\mb{g} $$
for a $c \in \mathbb{C}$.
However, according to Lemma \ref{lem:gg*}, this is only possible if $c=0$ and therefore, $\mb{v}$ itself belongs to $\ker(1-L_\mc{M})$.
By the definition of $\mb{v}$, this implies $(1-L_\mc{M})^m \mb{u}=0$, i.e., $\mb{u} \in \ker(1-L_\mc{M})^m$.
Invoking the induction hypothesis, we obtain $\mb{u} \in \ker(1-L_\mc{M})$ and thus, we have shown that $\ker (1-L_\mc{M})^{m+1} \subset \ker(1-L_\mc{M})$.
The converse inclusion $\ker(1-L_\mc{M}) \subset \ker(1-L_\mc{M})^{m+1}$ is trivial and we arrive at the claim $\ker(1-L_\mc{M})^m=\ker(1-L_\mc{M})$ for all $m \in \mathbb{N}$.

Finally, we show that $\mc{M}=\langle \mb{g} \rangle$.
To this end, let $\mb{u} \in \mc{M}$ be arbitrary.
Since $(1-L_\mc{M})$ is nilpotent, there exists an $m \in \mathbb{N}$ such that $(1-L_\mc{M})^m \mb{u}=0$, i.e., $\mb{u} \in \ker(1-L_\mc{M})^m$.
By the above, this implies $\mb{u} \in \ker(1-L_\mc{M})$ and invoking Lemma \ref{lem:specL2}, we infer the existence of a $c \in \mathbb{C}$ such that $\mb{u}=c\mb{g}$.
Thus, we have shown that $\mc{M} \subset \langle \mb{g} \rangle$.
Consequently, $\mc{M}$ is at most one--dimensional.
However, $\mc{M}$ cannot be zero--dimensional because this would contradict $\sigma(L_\mc{M})=\{1\}$.
This shows that $\dim \mc{M}=1$, $\mc{M}=\langle \mb{g} \rangle$ and the algebraic multiplicity (which equals $\dim \mc{M}$ by definition) is one.
\end{proof}

\subsection{The time evolution of linear perturbations}
In the previous two sections, we have decomposed the Hilbert space 
$\mc{H}$ into the direct sum $\mc{H}=\mc{M} \oplus \mc{N}$ where
the closed subspaces $\mc{M}$ and $\mc{N}$ are invariant under the time evolution given by the semigroup $S(\tau)$.
The operator $L_\mc{N}$, which acts on the stable subspace $\mc{N}$, has the same spectrum as $L$ except that the unstable gauge eigenvalue $1 \in \sigma_p(L)$ is missing.
Explicitly, we have
$\sigma(L_\mc{N}) \subset \{\lambda \in \mathbb{C}: \Re \lambda \leq \max\{-\frac{1}{2}, s_0\} \}$, cf.~Lemma \ref{lem:specL2}, and the unstable subspace $\mc{M}$ is one--dimensional and spanned by the gauge mode $\mb{g}$.
This is a very satisfactory situation which complies with heuristic expectations.
Our aim in this section is to obtain estimates on the semigroup $S_\mc{N}$ which describes the time evolution on the stable subspace $\mc{N}$.
In other words, we want to translate the spectral bounds we know for $L_\mc{N}$ into semigroup bounds for the associated time evolution.
Since $L_\mc{N}$ is not self--adjoint, this is nontrivial.
Luckily, there exists an abstract machinery which does most of the work for us, the Gearhardt--Pr\"uss--Hwang--Greiner Theorem, see e.g., \cite{engel}, p.~302, Theorem 1.11, or \cite{helffer}.
This theorem asserts that, if we can find an $\omega \in \mathbb{R}$ such that the resolvent $R_{L_\mc{N}}(\lambda)$ is uniformly bounded for all $\lambda \in \mathbb{C}$ with $\Re \lambda \geq \omega$, then the semigroup $S_\mc{N}$ satisfies the growth estimate $\|S_\mc{N}(\tau)\|\lesssim e^{\omega \tau}$ for all $\tau \geq 0$.
Naively, one might expect the estimate $\|S_\mc{N}(\tau)\|\lesssim e^{\max\{-\frac{1}{2}, s_0\}\tau}$ since $\sigma(L_\mc{N}) \subset \{\lambda \in \mathbb{C}: \Re \lambda \leq \max\{-\frac{1}{2}, s_0\} \}$.
Based on the 
Gearhardt--Pr\"uss--Hwang--Greiner Theorem, we will prove that, up to an $\varepsilon$--loss, this is indeed the case.

First, we obtain a bound for the resolvent $R_{L_0}(\lambda)$ of the free operator $L_0$ which is particularly useful far away from the real axis. 
In what follows, we denote by $[R_{L_0}(\lambda)\mb{f}]_j$, $j=1,2$, the individual components of $R_{L_0}(\lambda)\mb{f}$.
Furthermore, recall the definition of the bounded operator $K: H_2 \to H_1$ from the proof of Lemma \ref{lem:L'compact} which is given by
$$ (K u)(\rho)=\int_0^\rho \rho' u(\rho')d\rho'. $$

\begin{lemma}
 \label{lem:res0}
 Let $\varepsilon>0$ be fixed but arbitrary. Then the second component of the 
free resolvent satisfies the bound
 $$ \|K[R_{L_0}(\lambda)\mb{f}]_2\|_1 \lesssim \frac{\|\mb{f}\|}{|\lambda-2|} $$
 for all $\mb{f} \in \mc{H}$ and all $\lambda \in \mathbb{C}$ with $\Re \lambda \geq -\frac{1}{2}+\varepsilon$, $\lambda \not= 2$.
\end{lemma}

\begin{proof}
Let $\Re \lambda \geq -\frac{1}{2}+\varepsilon$.
According to Corollary \ref{cor:specL0}, $R_{L_0}(\lambda)$ exists and we have the estimate
$$ \|R_{L_0}(\lambda)\|\leq \frac{1}{\Re \lambda+\frac{1}{2}} $$
which shows that
\begin{equation} 
\label{eq:proofres0}
\|[R_{L_0}(\lambda)\mb{f}]_j\|_j \leq \|R_{L_0}(\lambda)\mb{f}\| \leq \frac{\|\mb{f}\|}{\Re \lambda+\frac{1}{2}} 
\end{equation}
for $j=1,2$ and all $\mb{f} \in \mc{H}$.
 By setting $V_1=0$ in Lemma \ref{lem:actionL}, we see that $(\lambda-L_0)\mb{u}=\mb{f}$ implies
$$ u_1(\rho)=\rho^2 u_2(\rho)+(\lambda-2)\int_0^\rho \rho' u_2(\rho')d\rho'-\int_0^\rho \rho'  f_2(\rho')d\rho' $$
and, in terms of resolvent components, this reads
$$ [R_{L_0}(\lambda)\mb{f}]_1(\rho)=\rho^2 [R_{L_0}(\lambda)\mb{f}]_2(\rho)+(\lambda-2)K[R_{L_0}(\lambda)\mb{f}]_2(\rho)-K f_2(\rho). $$
Consequently, we have
\begin{align*}
 |\lambda-2|^2\|K[R_{L_0}(\lambda)\mb{f}]_2\|_1^2 &\lesssim 
\|[R_{L_0}(\lambda)\mb{f}]_1\|_1^2+\int_0^1 \left |\rho^2 u_2'(\rho)+2\rho u_2(\rho) \right |^2 \frac{d\rho}{\rho^2}+\|Kf_2\|_1^2 \\
&\lesssim \frac{\|\mb{f}\|^2}{|\Re \lambda+\frac{1}{2}|^2}+\int_0^1 \rho^2 |u_2'(\rho)|^2 d\rho+\int_0^1 |u_2(\rho)|^2 d\rho+\|f_2\|_2^2 \\
&\lesssim \|\mb{f}\|^2+\|[R_{L_0}(\lambda)\mb{f}]_2\|_2^2+\|\mb{f}\|^2 \\
&\lesssim \|\mb{f}\|^2
\end{align*}
by Lemma \ref{lem:HinC} and the estimate \eqref{eq:proofres0}.
\end{proof}

Lemma \ref{lem:res0} is crucial in order to obtain estimates for the full resolvent $R_L(\lambda)$ far away from the real axis.

\begin{lemma}
 \label{lem:faraway}
 Let $\varepsilon>0$ be fixed but arbitrary. Assume that $\Re\lambda \geq -\frac{1}{2}+\varepsilon$ and $|\lambda|$ is sufficiently large.
 Then the resolvent $R_L(\lambda)$ exists and it is uniformly bounded as $|\lambda| \to \infty$.
 In particular, the operator $L$ does not have spectrum far away from the real axis in the right half--plane $\{\lambda \in \mathbb{C}: \Re \lambda \geq -\frac{1}{2}+\varepsilon\}$.
\end{lemma}

\begin{proof}
Let $\Re \lambda \geq -\frac{1}{2}+\varepsilon$.
By Corollary \ref{cor:specL0}, we see that $\lambda \in \rho(L_0)$ and therefore, we have the identity
$$ \lambda-L=(1-L'R_{L_0}(\lambda))(\lambda-L_0). $$
Now note that 
$$ L'R_{L_0}(\lambda) \mb{f}=\left ( \begin{array}{c} -V_1 K [R_{L_0}(\lambda)\mb{f}]_2 \\ 0 \end{array} \right ) $$
and thus, Lemma \ref{lem:res0} implies
$$ \|L' R_{L_0}(\lambda)\mb{f}\|=\|V_1 K [R_{L_0}(\lambda) \mb{f}]_2 \|_1 \lesssim \frac{\|\mb{f}\|}{|\lambda-2|} $$
for all $\mb{f} \in \mc{H}$.
This estimate implies that the Neumann series
$$ (1-L'R_{L_0}(\lambda))^{-1}=\sum_{k=0}^\infty (L'R_{L_0}(\lambda))^k $$
converges in the operator norm provided that $|\lambda|$ is sufficiently large.
As a consequence, 
$$ R_L(\lambda)=R_{L_0}(\lambda)(1-L'R_{L_0}(\lambda))^{-1} $$
exists for large $|\lambda|$ and is uniformly bounded as $|\lambda| \to \infty$ since
$$ \|R_{L_0}(\lambda)\|\leq \frac{1}{\mathrm{Re}\lambda+\frac{1}{2}} $$
by Corollary \ref{cor:specL0}.
In particular, there is no spectrum of $L$ far away from the real axis in $\{\lambda \in \mathbb{C}: \mathrm{Re}\lambda \geq -\frac{1}{2}+\varepsilon\}$.
\end{proof}

With these preparations we can prove the crucial estimate for the semigroup on the stable subspace which shows that mode stability of the fundamental self--similar solution implies its linear stability. This result concludes the linear perturbation theory.

\begin{theorem}[Mode stability implies linear stability]
\label{thm:linear}
Fix $\varepsilon>0$.
There exists a projection $P \in \mc{B}(\mc{H})$ onto $\langle \mb{g} \rangle$
which commutes with the semigroup $S(\tau)$ such that
$$ \|S(\tau)P\mb{f}\|=e^\tau \|P\mb{f}\| $$ as well as
$$ \|S(\tau)(1-P)\mb{f}\| \lesssim e^{(\max\{-
\frac{1}{2},s_0\}+\varepsilon)\tau}\|(1-P)\mb{f}\| $$
for all $\tau \geq 0$ and all $\mb{f} \in \mc{H}$
where $s_0$ is the spectral bound, see Definition \ref{def:sb}.
\end{theorem}

\begin{proof}
The existence of $P$ and the fact that $P$ commutes with $S(\tau)$ have already been discussed in Sec.~\ref{sec:gauge}.
Furthermore, by Lemma \ref{lem:M}, the unstable subspace $\mc{M}=P\mc{H}$ is spanned by the gauge mode $\mb{g}$.
Thus, for any $\mb{f} \in \mc{H}$, there exists a constant $c(\mb{f}) \in \mathbb{C}$, depending on $\mb{f}$, such that $P\mb{f}=c(\mb{f})\mb{g}$.
Consequently, $|c(\mb{f})|=\frac{\|P\mb{f}\|}{\|\mb{g}\|}$ and applying the time evolution yields
$$ \|S(\tau)P\mb{f}\|=|c(\mb{f})|\|S(\tau)\mb{g}\|=
e^{\tau}|c(\mb{f})|\|\mb{g}\|=e^{\tau}\|P\mb{f}\| $$
since $\mb{g}$ is an eigenfunction of the generator $L$ with eigenvalue $1$, see Lemma \ref{lem:specL2}.

In order to prove the estimate on the stable subspace, note that the statement from Lemma \ref{lem:faraway} also holds true for the operator $L_\mc{N}$ since $R_{L_\mc{N}}(\lambda)$ is the restriction of $R_L(\lambda)$ to the stable subspace $\mc{N}=(1-P)\mc{H}$.
As already mentioned, the spectrum of $L_\mc{N}$ satisfies $\sigma(L_\mc{N})\subset \{\lambda \in \mathbb{C}: \mathrm{Re}\lambda\leq \max\{-\frac{1}{2},s_0\}\}$.
This and Lemma \ref{lem:faraway} imply that the resolvent $R_{L_\mc{N}}(\lambda)$ is uniformly bounded for $\Re \lambda \geq \max\{-\frac{1}{2},s_0\}+\varepsilon$, i.e., there exists a constant $C>0$ such that
$$ \|R_{L_\mc{N}}(\lambda)\|\leq C $$
for all $\lambda \in \mathbb{C}$ with $\Re \lambda \geq \max\{-\frac{1}{2},s_0\}+\varepsilon$.
Thus, the Gearhardt--Pr\"uss--Hwang--Greiner Theorem (\cite{engel}, p.~302, Theorem 1.11, see also the recent \cite{helffer}) implies the claim.
\end{proof}

\bibliography{wmlin}{}
\bibliographystyle{plain}

\end{document}